\documentclass[final,3p,times]{elsarticle}

\usepackage{amssymb}
\usepackage{amsmath}
\usepackage{amsthm}
\usepackage{graphicx}

\usepackage[colorlinks=true]{hyperref}					

\newcommand{\bbK}{\mathbb{K}}
\newcommand{\bbR}{\mathbb{R}}

\newcommand{\bbH}{\mathbb{H}}

\newcommand{\calI}{\mathcal{I}}
\newcommand{\calJ}{\mathcal{J}}
\newcommand{\rmc}{\mathrm{c}}

\newcommand{\rmI}{\mathrm{I}}
\newcommand{\rmN}{\mathrm{N}}
\newcommand{\rmT}{\mathrm{T}}
\newcommand{\Tr}{\mathrm{Tr}}

\newcommand{\sgn}{\mathrm{sgn}}

\newcommand{\abs}[1]{|{#1}|}

\newcommand{\biggparen}[1]{\biggl({#1}\biggr)}

\newcommand{\bigbracket}[1]{\bigl[{#1}\bigr]}

\newcommand{\biggbracket}[1]{\biggl[{#1}\biggr]}

\newcommand{\set}[1]{\{{#1}\}}

\newcommand{\norm}[1]{\|{#1}\|}

\newcommand{\ip}[2]{\langle{#1},{#2}\rangle}

\theoremstyle{plain}
\newtheorem{theorem}{Theorem}
\newtheorem{lemma}[theorem]{Lemma}

\newtheorem*{SH}{Schur-Horn Theorem}

\theoremstyle{definition}
\newtheorem{definition}[theorem]{Definition}

\newtheorem{example}[theorem]{Example}

\begin{document}
\begin{frontmatter}
\title{Constructing finite frames of a given spectrum and set of lengths}

\author[MO]{Jameson Cahill}
\author[AFIT]{Matthew Fickus}
\ead{Matthew.Fickus@afit.edu}
\author[Princeton]{Dustin G.~Mixon}
\author[AFIT]{Miriam J.~Poteet}
\author[MD]{Nathaniel K. Strawn}

\address[MO]{Department of Mathematics, University of Missouri, Columbia, MO 65211, USA}
\address[AFIT]{Department of Mathematics and Statistics, Air Force Institute of Technology, Wright-Patterson Air Force Base, OH 45433, USA}
\address[Princeton]{Program in Applied and Computational Mathematics, Princeton University, Princeton, NJ 08544, USA}
\address[MD]{Department of Mathematics, University of Maryland, College Park, MD 20742, USA}

\begin{abstract}
When constructing finite frames for a given application, the most important consideration is the spectrum of the frame operator.  Indeed, the minimum and maximum eigenvalues of the frame operator are the optimal frame bounds, and the frame is tight precisely when this spectrum is constant.  Often, the second-most important design consideration is the lengths of frame vectors: Gabor, wavelet, equiangular and Grassmannian frames are all special cases of equal norm frames, and unit norm tight frame-based encoding is known to be optimally robust against additive noise and erasures.  We consider the problem of constructing frames whose frame operator has a given spectrum and whose vectors have prescribed lengths.  For a given spectrum and set of lengths, the existence of such frames is characterized by the Schur-Horn Theorem---they exist if and only if the spectrum majorizes the squared lengths---the classical proof of which is nonconstructive.   Certain construction methods, such as harmonic frames and spectral tetris, are known in the special case of unit norm tight frames, but even these provide but a few examples from the manifold of all such frames, the dimension of which is known and nontrivial.  In this paper, we provide a new method for explicitly  constructing any and all frames whose frame operator has a prescribed spectrum and whose vectors have prescribed lengths.  The method itself has two parts.  In the first part, one chooses eigensteps---a sequence of interlacing spectra---that transform the trivial spectrum into the desired one.  The second part is to explicitly compute the frame vectors in terms of these eigensteps; though nontrivial, this process is nevertheless straightforward enough to be implemented by hand, involving only arithmetic, square roots and matrix multiplication.
\end{abstract}

\end{frontmatter}


\section{Introduction}
Letting $\bbK$ be either the real or complex field, the \textit{synthesis operator} of a sequence of vectors $F=\set{f_n}_{n=1}^N$ in an $M$-dimensional Hilbert space $\bbH_M$ over $\bbK$ is $F:\mathbb{K}^N\rightarrow\bbH_M$, $\smash{Fg:=\sum_{n=1}^N g(n)f_n}$.  Viewing $\bbH_M$ as $\bbK^M$, $F$ is the $M\times N$ matrix whose columns are the $f_n$'s.  Note that here and throughout, we make no notational distinction between the vectors themselves and the synthesis operator they induce.  The vectors $F$ are said to be a \textit{frame} for $\bbH_M$ if there exists \textit{frame bounds} $0<A\leq B<\infty$ such that $A\norm{f}^2\leq\norm{F^*f}^2\leq B\norm{f}^2$ for all $f\in\bbH_M$.  In this finite-dimensional setting, the optimal frame bounds $A$ and $B$ of an arbitrary $F$ are the least and greatest eigenvalues of the \textit{frame operator}:
\begin{equation}
\label{equation.definition of frame operator}
FF^*=\sum_{n=1}^{N}f_nf_n^*,
\end{equation}
respectively.  Here, $f_n^*$ is the linear functional $f_n^*:\bbH_M\rightarrow\bbK$, $f_n^*f:=\ip{f}{f_n}$.  In particular, we have that $F$ is a frame if and only if the $f_n$'s span $\bbH_M$, which necessitates $M\leq N$.

Frames provide numerically stable methods for finding overcomplete decompositions of vectors, and as such are useful tools in various signal processing applications~\cite{KovacevicC:07a,KovacevicC:07b}.  Indeed, if $F$ is a frame, then any $f\in\bbH_M$ can be decomposed as
\begin{equation}
\label{equation.frame expansion}
f=F\tilde{F}^*f=\sum_{n=1}^{N}\ip{f}{\tilde{f}_n}f_n,
\end{equation}
where $\tilde{F}=\set{\tilde{f}_n}_{n=1}^{N}$ is a \textit{dual frame} of $F$, meaning it satisfies $F\tilde{F}^*=\rmI$.  The most often-used dual frame is the \textit{canonical} dual, namely the pseudoinverse $\tilde{F}=(FF^*)^{-1}F$.  Note that computing a canonical dual involves the inversion of the frame operator.  As such, when designing a frame for a given application, it is important to retain control over the spectrum $\set{\lambda_m}_{m=1}^{M}$ of $FF^*$.  Here and throughout, such spectra are arranged in nonincreasing order, with the optimal frame bounds $A$ and $B$ being $\lambda_M$ and $\lambda_1$, respectively.

Of particular interest are \textit{tight frames}, namely frames for which $A=B$.  Note this occurs precisely when $\lambda_m=A$ for all $m$, meaning $FF^*=A\rmI$.  In this case, the canonical dual is given by $\tilde{f}_n=\frac1A f_n$, and~\eqref{equation.frame expansion} becomes an overcomplete generalization of an orthonormal basis decomposition.  Tight frames are not hard to construct: we simply need the rows of $F$ to be orthogonal and have constant squared norm $A$.  However, this problem becomes significantly more difficult if we further require the $f_n$'s---the columns of $F$---to have prescribed lengths.

In particular, much attention has been paid to the problem of constructing \textit{unit norm tight frames} (UNTFs): tight frames for which $\norm{f_n}=1$ for all $n$.  Here, since $MA=\Tr(FF^*)=\Tr(F^*F)=N$, we see that $A$ is necessarily $\frac NM$.  UNTFs are known to be optimally robust with respect to additive noise~\cite{GoyalVT:98} and erasures \cite{CasazzaK:03,HolmesP:04}.  Moreover, all unit norm sequences $F$ satisfy the zeroth-order \textit{Welch bound} \smash{$\Tr[(FF^*)^2]\geq\frac{N^2}{M}$}, which is achieved precisely when $F$ is a UNTF~\cite{Waldron:03,Welch:74}; a physics-inspired interpretation of this fact leading to an optimization-based proof of existence of UNTFs is given in~\cite{BenedettoF:03}.  We further know that such frames are commonplace: when $N\geq M+1$, the manifold of all $M\times N$ real UNTFs, modulo rotations, is known to have dimension $(M-1)(N-M-1)$~\cite{DykemaS:06}.  Essentially, when $N=M+1$, this manifold is zero-dimensional since the only UNTFs are regular simplices~\cite{GoyalKK:01}; each additional unit norm vector injects $M-1$ additional degrees of freedom into this manifold, in accordance with the dimension of the unit sphere in $\bbR^M$.  Local parametrizations of this manifold are given in~\cite{Strawn:11}.  The \textit{Paulsen problem} involves projecting a given frame onto this manifold, and differential calculus-based methods for doing so are given in~\cite{BodmannC:10,CasazzaFM:11}.

In light of these facts, it is surprising to note how few explicit constructions of UNTFs are known.  Indeed, a constructive characterization of all UNTFs is only known for $M=2$~\cite{GoyalKK:01}.  For arbitrary $M$ and $N$, there are only two known general construction techniques: truncations of discrete Fourier transform matrices known as \textit{harmonic frames}~\cite{GoyalKK:01} and a sparse construction method dubbed \textit{spectral tetris}~\cite{CasazzaFMWZ:11}.  To emphasize this point, we note that there are only a small finite number of known constructions of $3\times 5$ UNTFs, despite the fact that an infinite number of such frames exist even modulo rotations, their manifold being of dimension $(M-1)(N-M-1)=2$.  The reason for this is that in order to construct a UNTF, one must solve a large system of quadratic equations in many variables: the columns of $F$ must have unit norm, and the rows of $F$ must be orthogonal with constant norm \smash{$(\frac NM)^{\frac12}$}.

In this paper, we show how to explicitly construct all UNTFs, and moreover, how to explicitly construct every frame whose frame operator has a given arbitrary spectrum and whose vectors are of given arbitrary lengths.  To do so, we build on the existing theory of majorization and the Schur-Horn Theorem.  To be precise, given two nonnegative nonincreasing sequences $\set{\lambda_n}_{n=1}^N$ and $\set{\mu_n}_{n=1}^{N}$, we say that $\set{\lambda_n}_{n=1}^N$ \textit{majorizes} $\set{\mu_n}_{n=1}^{N}$, denoted $\set{\lambda_n}_{n=1}^N\succeq\set{\mu_n}_{n=1}^{N}$, if
\begin{align*}
\sum_{n'=1}^{n}\lambda_{n'}
&\geq\sum_{n'=1}^{n}\mu_{n'}\qquad\forall n=1,\dotsc,N-1,\\
\sum_{n'=1}^{N}\lambda_{n'}
&=\sum_{n'=1}^{N}\mu_{n'}.
\end{align*}
Viewed as discrete functions over the axis $\set{1,\dotsc,N}$, having $\set{\lambda_n}_{n=1}^N$ majorize $\set{\mu_n}_{n=1}^{N}$ means that the total area under both curves is equal, and that the area under $\set{\lambda_n}_{n=1}^N$ is distributed more to the left than that of $\set{\mu_n}_{n=1}^{N}$.  A classical result of Schur~\cite{Schur:23} states that the spectrum of a self-adjoint positive semidefinite matrix necessarily majorizes its diagonal entries.  A few decades later, Horn gave a nonconstructive proof of a converse result~\cite{Horn:54}, showing that if $\set{\lambda_n}_{n=1}^N\succeq\set{\mu_n}_{n=1}^{N}$, then there exists a self-adjoint matrix that has $\set{\lambda_n}_{n=1}^N$ as its spectrum and $\set{\mu_n}_{n=1}^{N}$ as its diagonal.  These two results are collectively known as the Schur-Horn Theorem:
\begin{SH}
There exists a positive semidefinite self-adjoint matrix with spectrum $\set{\lambda_n}_{n=1}^{N}$ and diagonal entries $\set{\mu_n}_{n=1}^{N}$ if and only if $\set{\lambda_n}_{n=1}^N\succeq\set{\mu_n}_{n=1}^{N}$.
\end{SH}
Over the years, several methods for explicitly constructing Horn's matrices have been found; see~\cite{DhillonHST:05} for a nice overview.  Many current methods rely on Givens rotations~\cite{CasazzaL:06,DhillonHST:05,ViswanathA:99}, while others involve optimization~\cite{Chu:95}.  With regards to frame theory, the significance of the Schur-Horn Theorem is that it completely characterizes whether or not there exists a frame whose frame operator has a given spectrum and whose vectors have given lengths; this follows from applying it to the \textit{Gram matrix} $F^*F$, whose diagonal entries are the values $\set{\norm{f_n}^2}_{n=1}^{N}$ and whose spectrum $\set{\lambda_n}_{n=1}^{N}$ is a zero-padded version of the spectrum $\set{\lambda_m}_{m=1}^{M}$ of the frame operator $FF^*$.  Indeed, majorization inequalities arose during the search for tight frames with given lengths~\cite{CasazzaFKLT:06,DykemaFKLOW:03}, and the explicit connection between frames and the Schur-Horn Theorem is noted in \cite{AntezanaMRS:07,TroppDHS:05}.  This connection was then exploited to solve various frame theory problems, such as frame completion~\cite{MasseyR:08}.

In this paper, we follow the approach of~\cite{HornJ:85} in which majorization is viewed as the end result of the repeated application of a more basic idea: eigenvalue interlacing.  To be precise, a nonnegative nonincreasing sequence $\set{\gamma_m}_{m=1}^{M}$ \textit{interlaces} on another such sequence $\set{\beta_m}_{m=1}^{M}$, denoted $\set{\beta_m}_{m=1}^{M}\sqsubseteq\set{\gamma_m}_{m=1}^{M}$, provided
\begin{equation}
\label{equation.definition of interlacing}
\beta_M\leq\gamma_M\leq\beta_{M-1}\leq\gamma_{M-1}\leq\dots\leq\beta_2\leq\gamma_2\leq\beta_1\leq\gamma_1.
\end{equation}
Under the convention $\gamma_{M+1}:=0$, we have that $\set{\beta_m}_{m=1}^{M}\sqsubseteq\set{\gamma_m}_{m=1}^{M}$ if and only if $\gamma_{m+1}\leq\beta_m\leq\gamma_m$ for all $m=1,\dotsc,M$.  Interlacing arises in the context of frame theory by considering partial sums of the frame operator~\eqref{equation.definition of frame operator}.  To be precise, given any sequence of vectors $F=\set{f_n}_{n=1}^{N}$ in $\bbH_M$, then for every $n=1,\dotsc,N$, we consider the partial sequence of vectors $F_n:=\set{f_{n'}^{}}_{n'=1}^{n}$.    Note that $F_N^{}=F$ and the frame operator of $F_n^{}$ is
\begin{equation}
\label{equation.definition of partial frame operator}
F_n^{}F_n^*=\sum_{n'=1}^{n}f_{n'}^{}f_{n'}^*.
\end{equation}
Let $\set{\lambda_{n;m}}_{m=1}^{M}$ denote the spectrum of~\eqref{equation.definition of partial frame operator}.  For any $n=1,\dotsc,N-1$,~\eqref{equation.definition of partial frame operator} gives that $F_{n+1}^{}F_{n+1}^*=F_n^{}F_n^*+f_{n+1}^{}f_{n+1}^*$  and so a classical result~\cite{HornJ:85} involving the addition of rank-one positive operators gives that $\set{\lambda_{n;m}}_{m=1}^{M}\sqsubseteq\set{\lambda_{n+1;m}}_{m=1}^{M}$.  Moreover, if $\norm{f_n}^2=\mu_n$ for all $n=1,\dotsc,N$, then for any such $n$,
\begin{equation}
\label{equation.trace condition}
\sum_{m=1}^M\lambda_{n;m}
=\Tr(F_n^{}F_n^*)
=\Tr(F_n^*F_n^{})
=\sum_{n'=1}^{n}\norm{f_{n'}}^2
=\sum_{n'=1}^n\mu_{n'}.
\end{equation}
Note that as $n$ increases, the Gram matrix grows in dimension but the frame operator does not since $F_n^*F_n^{}:\bbK^n\rightarrow\bbK^n$ but $F_n^{}F_n^*:\bbH_M\rightarrow\bbH_M$.  We call a sequence of interlacing spectra that satisfy~\eqref{equation.trace condition} a sequence of \textit{eigensteps}:
\begin{definition}
\label{definition.eigensteps}
Given nonnegative nonincreasing sequences $\set{\lambda_m}_{m=1}^{M}$ and $\set{\mu_n}_{n=1}^{N}$, a sequence of \textit{eigensteps} is a doubly-indexed sequence of sequences $\set{\set{\lambda_{n;m}}_{m=1}^{M}}_{n=0}^{N}$ for which:
\begin{enumerate}
\renewcommand{\labelenumi}{(\roman{enumi})}
\item The initial sequence is trivial:
\begin{equation*}
\lambda_{0;m}=0\qquad\forall m=1,\dotsc,M.
\end{equation*}
\item The final sequence is $\set{\lambda_m}_{m=1}^{M}$:
\begin{equation*}
\lambda_{N;m}=\lambda_m\qquad\forall m=1,\dotsc,M.
\end{equation*}
\item
The sequences interlace:
\begin{equation*}
\set{\lambda_{n-1;m}}_{m=1}^{M}\sqsubseteq\set{\lambda_{n;m}}_{m=1}^{M}\qquad\forall n=1,\dotsc,N.
\end{equation*}
\item
The trace condition is satisfied:
\begin{equation*}
\sum_{m=1}^{M}\lambda_{n;m}=\sum_{n'=1}^{n}\mu_{n'}\qquad\forall n=1,\dotsc,N.
\end{equation*}
\end{enumerate}
\end{definition}
As we have just discussed, every sequence of vectors whose frame operator has the spectrum $\set{\lambda_m}_{m=1}^{N}$ and whose vectors have squared lengths $\set{\mu_n}_{n=1}^{N}$ generates a sequence of eigensteps.  In the next section, we adapt a proof technique of~\cite{HornJ:85} to show the converse is true.  Specifically, Theorem~\ref{theorem.necessity and sufficiency of eigensteps} characterizes and proves the existence of sequences of vectors that generate a given sequence of eigensteps.  In Section~3, we then use this characterization to provide an algorithm for explicitly constructing all such sequences of vectors; see Theorem~\ref{theorem.explicit frame construction}.  Though nontrivial, this algorithm is nevertheless straightforward enough to be implemented by hand in small-dimensional examples, involving only arithmetic, square roots and matrix multiplication.  We will see that once the eigensteps have been chosen, the algorithm gives little freedom in picking the frame vectors themselves.  That is, modulo rotations, the eigensteps are the free parameters when designing a frame whose frame operator has a given spectrum and whose vectors have given lengths.  

The significance of these methods is that they explicitly construct every possible finite frame of a given spectrum and set of lengths.  Computing the Gram matrices of such frames produces every possible matrix that satisfies the Schur-Horn Theorem; previous methods have only constructed a subset of such matrices.  Moreover, in the special case where the spectrums and lengths are constant, these methods construct every equal norm tight frame.  This helps narrow the search for frames we want for applications: tight Gabor, wavelet, equiangular and Grassmannian frames.


\section{The necessity and sufficiency of eigensteps}

\noindent The purpose of this section is to prove the following result:
\begin{theorem}
\label{theorem.necessity and sufficiency of eigensteps}
For any nonnegative nonincreasing sequences $\set{\lambda_m}_{m=1}^{M}$ and $\set{\mu_n}_{n=1}^{N}$, every sequence of vectors $F=\set{f_n}_{n=1}^{N}$ in $\bbH_M$ whose frame operator $FF^*$ has spectrum $\set{\lambda_m}_{m=1}^{M}$ and which satisfies $\norm{f_n}^2=\mu_n$ for all $n$ can be constructed by the following process:
\begin{enumerate}
\renewcommand{\labelenumi}{\Alph{enumi}.}
\item
Pick eigensteps $\set{\set{\lambda_{n;m}}_{m=1}^{M}}_{n=0}^{N}$ as in Definition~\ref{definition.eigensteps}. 
\item
For each $n=1,\dotsc,N$, consider the polynomial:
\begin{equation}
\label{equation.necessity and sufficiency of eigensteps 1}
p_n(x):=\prod_{m=1}^{M}(x-\lambda_{n;m}).
\end{equation}
Take any $f_1\in\bbH_M$ such that $\norm{f_1}^2=\mu_1$.  For each $n=1,\dotsc,N-1$, choose any $f_{n+1}$ such that
\begin{equation}
\label{equation.necessity and sufficiency of eigensteps 2}
\norm{P_{n;\lambda}f_{n+1}}^2=-\lim_{x\rightarrow\lambda}(x-\lambda)\frac{p_{n+1}(x)}{p_n(x)}
\end{equation}
for all $\lambda\in\set{\lambda_{n;m}}_{m=1}^M$, where $P_{n;\lambda}$ denotes the orthogonal projection operator onto the eigenspace $\rmN(\lambda\rmI-F_n^{}F_n^*)$ of the frame operator $F_n^{}F_n^*$ of $F_n=\set{f_{n'}}_{n'=1}^{n}$.  The limit in~\eqref{equation.necessity and sufficiency of eigensteps 2} exists and is nonpositive.
\end{enumerate}
Conversely, any $F$ constructed by this process has $\set{\lambda_m}_{m=1}^{M}$ as the spectrum of $FF^*$ and $\norm{f_n}^2=\mu_n$ for all $n$.

\noindent Moreover, for any $F$ constructed in this manner, the spectrum of $F_n^{}F_n^*$ is $\set{\lambda_{n;m}}_{m=1}^{M}$ for all $n=1,\dotsc,N$.
\end{theorem}
We note that as it stands, Theorem~\ref{theorem.necessity and sufficiency of eigensteps} is not an easily-implementable algorithm, as Step~A requires one to select a valid sequence of eigensteps---not an obvious feat---while Step~B requires one to compute orthonormal eigenbases for each $F_n$.  These concerns will be addressed in the following section.  We further note that Theorem~\ref{theorem.necessity and sufficiency of eigensteps} only claims to construct all possible such $F$, sidestepping the issue of whether such an $F$ actually exists for a given $\set{\lambda_m}_{m=1}^{M}$ and $\set{\mu_n}_{n=1}^{N}$.  This issue is completely resolved by the Schur-Horn Theorem.  Indeed, in the case where $M\leq N$, \cite{AntezanaMRS:07} shows that there exists a sequence of vectors $F=\set{f_n}_{n=1}^{N}$ in $\bbH_M$ whose frame operator $FF^*$ has spectrum $\set{\lambda_m}_{m=1}^{M}$ and which satisfies $\norm{f_n}^2=\mu_n$ for all $n$ if and only if $\set{\lambda_m}_{m=1}^{M}\cup\set{0}_{n=M+1}^{N}\succeq\set{\mu_n}_{n=1}^{N}$.  In the case where $M>N$, a similar argument shows that such a sequence of vectors exists if and only if $\set{\lambda_m}_{m=1}^{N}\succeq\set{\mu_n}_{n=1}^{N}$ and $\lambda_m=0$ for all $m=N+1,\dotsc,M$.  As Step~B of Theorem~\ref{theorem.necessity and sufficiency of eigensteps} can always be completed for any valid sequence of eigensteps, these majorization conditions in fact characterize those values $\set{\lambda_m}_{m=1}^{M}$ and $\set{\mu_n}_{n=1}^{N}$ for which Step~A can successfully be performed; we leave a deeper exploration of this fact for future work.  In order to prove Theorem~\ref{theorem.necessity and sufficiency of eigensteps}, we first obtain some supporting results.  The following lemma gives a first taste of the connection between eigensteps and our frame construction problem:
\begin{lemma}
\label{lemma.eigensteps yield desired properties}
Let $\set{\lambda_m}_{m=1}^{M}$ and $\set{\mu_n}_{n=1}^{N}$ be nonnegative and nonincreasing, and let $\set{\set{\lambda_{n;m}}_{m=1}^{M}}_{n=0}^{N}$ be any corresponding sequence of eigensteps as in Definition~\ref{definition.eigensteps}.  If a sequence of vectors $F=\set{f_n}_{n=1}^{N}$ has the property that the spectrum of the frame operator $F_n^{}F_n^*$ of $F_n=\set{f_{n'}}_{n'=1}^n$ is $\set{\lambda_{n;m}}_{m=1}^{M}$ for all $n=1,\dotsc,N$, then the spectrum of $FF^*$ is $\set{\lambda_m}_{m=1}^{M}$ and $\norm{f_n}^2=\mu_n$ for all $n=1,\dotsc,N$.
\end{lemma}

\begin{proof}
Definition~\ref{definition.eigensteps}(ii) immediately gives that the spectrum of $FF^*=F_N^{}F_N^{*}$ is indeed $\set{\lambda_{n}}_{m=1}^{M}=\set{\lambda_{N;m}}_{m=1}^{M}$, as claimed.  Moreover, for any $n=1,\dotsc,N$, Definition~\ref{definition.eigensteps}(iv) gives
\begin{equation}
\label{equation.proof of explicit frame construction 0}
\sum_{n'=1}^{n}\norm{f_{n'}}^2
=\Tr(F_n^*F_n^{})
=\Tr(F_n^{}F_n^*)
=\sum_{m=1}^{M}\lambda_{n;m}
=\sum_{n'=1}^{n}\mu_{n'}.
\end{equation}
Letting $n=1$ in \eqref{equation.proof of explicit frame construction 0} gives $\norm{f_1}^2=\mu_1$, while for $n=2,\dotsc,N$, considering \eqref{equation.proof of explicit frame construction 0} at both $n$ and $n-1$ gives
\begin{equation*}
\norm{f_n}^2
=\sum_{n'=1}^{n}\norm{f_{n'}}^2-\sum_{n'=1}^{n-1}\norm{f_{n'}}^2
=\sum_{n'=1}^{n}\mu_{n'}-\sum_{n'=1}^{n-1}\mu_{n'}
=\mu_n.\qedhere
\end{equation*}
\end{proof}

The next result gives conditions that a vector must satisfy in order for it to perturb the spectrum of a given frame operator in a desired way, and was inspired by the proof of Theorem~4.3.10 in~\cite{HornJ:85}.
\begin{theorem}
\label{theorem.necessary lengths of projections}
Let $F_n=\set{f_{n'}}_{n'=1}^{n}$ be an arbitrary sequence of vectors in $\bbH_M$ and let $\set{\lambda_{n;m}}_{m=1}^{M}$ denote the eigenvalues of the corresponding frame operator $F_{n}^{}F_{n}^{*}$.  For any choice of $f_{n+1}$ in $\bbH_M$, let $F_{n+1}=\set{f_{n'}}_{n'=1}^{n+1}$. Then for any $\lambda\in\set{\lambda_{n;m}}_{m=1}^{M}$, the norm of the projection of $f_{n+1}$ onto the eigenspace $\rmN(\lambda\rmI-F_{n}^{}F_{n}^{*})$ is given by
\begin{equation*}
\norm{P_{n;\lambda}f_{n+1}}^2
=-\lim_{x\rightarrow\lambda}(x-\lambda)\frac{p_{n+1}(x)}{p_n(x)},
\end{equation*}
where $p_n(x)$ and $p_{n+1}(x)$ denote the characteristic polynomials of $F_{n}^{}F_{n}^{*}$ and $F_{n+1}^{}F_{n+1}^{*}$, respectively.
\end{theorem}
\begin{proof}
For the sake of notational simplicity,  let $F_n=F$, $f_{n+1}=f$, $F_{n+1}=G$, $P_{n;\lambda}=P_\lambda$, $p_{n}(x)=p(x)$, $p_{n+1}(x)=q(x)$, and let $\lambda_{n;m}=\beta_m$ for all $m=1,\dotsc,M$.  We will also use $\rmI$ to denote the identity matrix, and its dimension will be apparent from context.  To obtain the result, we will express the characteristic polynomial $\tilde{q}(x)$ of the $(n+1)\times(n+1)$ Gram matrix $G^*G$ in terms of the characteristic polynomial $\tilde{p}(x)$ of the $n\times n$ Gram matrix $F^*F$.  
Written in terms of their standard matrix representations, we have $\smash{G=\begin{bmatrix}F&f\end{bmatrix}}$, and so
\begin{equation}
\label{equation.proof of necessary lengths of projections 1}
G^*G
=\begin{bmatrix}F^*\\f^*\end{bmatrix}\begin{bmatrix}F&f\end{bmatrix}
=\begin{bmatrix}F^*F&F^*f\\f^*F&\norm{f}^2\end{bmatrix}.
\end{equation}
To compute the determinant of $x\rmI-G^*G$, it is helpful to compute the singular value decomposition $F=U\Sigma V^*$, and note that for any $x$ not in the diagonal of $\Sigma^*\Sigma$, the following matrix $W$ has unimodular determinant:
\begin{equation}
\label{equation.proof of necessary lengths of projections 2}
W
:=\begin{bmatrix}V&0\\0&1\end{bmatrix}\begin{bmatrix}\rmI&(x\rmI-\Sigma^*\Sigma)^{-1}V^*F^*f\\0&1\end{bmatrix}
=\begin{bmatrix}V&V(x\rmI-\Sigma^*\Sigma)^{-1}V^*F^*f\\0&1\end{bmatrix}.
\end{equation}
Subtracting~\eqref{equation.proof of necessary lengths of projections 1} from $x\rmI$ and conjugating  by \eqref{equation.proof of necessary lengths of projections 2} yields
\begin{align}
\nonumber
W^*(x\rmI-G^*G)W
&=\begin{bmatrix}V^*&0\\(V(x\rmI-\Sigma^*\Sigma)^{-1}V^*F^*f)^*&1\end{bmatrix}\begin{bmatrix}x\rmI-F^*F&-F^*f\\-f^*F&x-\|f\|^2\end{bmatrix}\begin{bmatrix}V&V(x\rmI-\Sigma^*\Sigma)^{-1}V^*F^*f\\0&1\end{bmatrix}\\
\label{equation.proof of necessary lengths of projections 3}
&=\begin{bmatrix}V^*&0\\f^*FV(x\rmI-\Sigma^*\Sigma)^{-1}V^*&1\end{bmatrix}\begin{bmatrix}(x\rmI-F^*F)V&(x\rmI-F^*F)V(x\rmI-\Sigma^*\Sigma)^{-1}V^*F^*f-F^*f\\-f^*FV&x-\|f\|^2-f^*FV(x\rmI-\Sigma^*\Sigma)^{-1}V^*F^*f\end{bmatrix}.
\end{align}
Since $F^*F=V\Sigma^*\Sigma V^*$ then $x\rmI-F^*F=x\rmI-V\Sigma^*\Sigma V^*=V(x\rmI-\Sigma^*\Sigma)V^*$.  As such, 
\begin{equation}
\label{equation.proof of necessary lengths of projections 4}
(x\rmI-F^*F)V(x\rmI-\Sigma^*\Sigma)^{-1}V^*F^*f-F^*f
=V(x\rmI-\Sigma^*\Sigma)V^*V(x\rmI-\Sigma^*\Sigma)^{-1}V^*F^*f-F^*f \\
=F^*f-F^*f
=0.
\end{equation}
Substituting~\eqref{equation.proof of necessary lengths of projections 4} into~\eqref{equation.proof of necessary lengths of projections 3} and again noting $V^*(x\rmI-F^*F)V=x\rmI-\Sigma^*\Sigma$ gives
\begin{align}
\nonumber
W^*(x\rmI-G^*G)W
&=\begin{bmatrix}V^*&0\\f^*FV(x\rmI-\Sigma^*\Sigma)^{-1}V^*&1\end{bmatrix}\begin{bmatrix}(x\rmI-F^*F)V&0\\-f^*FV&x-\|f\|^2-f^*FV(x\rmI-\Sigma^*\Sigma)^{-1}V^*F^*f\end{bmatrix}\\
\nonumber
&=\begin{bmatrix}V^*(x\rmI-F^*F)V&0\\f^*FV(x\rmI-\Sigma^*\Sigma)^{-1}V^*(x\rmI-F^*F)V-f^*FV&x-\|f\|^2-f^*FV(x\rmI-\Sigma^*\Sigma)^{-1}V^*F^*f\end{bmatrix}\\
\label{equation.proof of necessary lengths of projections 5}
&=\begin{bmatrix}x\rmI-\Sigma^*\Sigma&0\\0&x-\|f\|^2-f^*FV(x\rmI-\Sigma^*\Sigma)^{-1}V^*F^*f\end{bmatrix}.
\end{align}
Since $W$ has unimodular determinant, \eqref{equation.proof of necessary lengths of projections 5} implies
\begin{equation}
\label{equation.proof of necessary lengths of projections 6}
\tilde{q}(x)
:=\mathrm{det}(x\rmI-G^*G)\\
=\mathrm{det}\bigbracket{W^*(x\rmI-G^*G)W}\\
=\mathrm{det}(x\rmI-\Sigma^*\Sigma)(x-\|f\|^2-f^*FV(x\rmI-\Sigma^*\Sigma)^{-1}V^*F^*f).
\end{equation}
To simplify \eqref{equation.proof of necessary lengths of projections 6}, note that since $V$ is unitary,
\begin{equation}
\label{equation.proof of necessary lengths of projections 7}
\tilde{p}(x)
:=\mathrm{det}(x\rmI-F^*F)
=\mathrm{det}\bigbracket{V^*(x\rmI-F^*F)V}
=\det(x\rmI-\Sigma^*\Sigma).
\end{equation}
Moreover, letting $(\Sigma^*\Sigma)(n',n')$ denote the $n'$th diagonal entry of $\Sigma^*\Sigma$ yields
\begin{equation}
\label{equation.proof of necessary lengths of projections 8}
f^*FV(x\rmI-\Sigma^*\Sigma)^{-1}V^*F^*f
=(V^*F^*f)^*(x\rmI-\Sigma^*\Sigma)^{-1}(V^*F^*f)
=\sum_{n'=1}^{n}\frac{\abs{(V^*F^*f)(n')}^2}{x-(\Sigma^*\Sigma)(n',n')}.
\end{equation}
Substituting~\eqref{equation.proof of necessary lengths of projections 7} and~\eqref{equation.proof of necessary lengths of projections 8} into~\eqref{equation.proof of necessary lengths of projections 6} gives
\begin{equation}
\label{equation.proof of necessary lengths of projections 9}
\tilde{q}(x)
=\tilde{p}(x)\,\biggparen{x-\|f\|^2-\sum_{n'=1}^{n}\frac{\abs{(V^*F^*f)(n')}^2}{x-(\Sigma^*\Sigma)(n',n')}}\,.\\
\end{equation}
To continue simplifying \eqref{equation.proof of necessary lengths of projections 9}, let $\delta_{n'}$ denote the $n'$th standard basis element.  Then $V^*F^*=\Sigma^*U^*$ implies that for any $n'=1,\dotsc,n$,
\begin{equation}
\label{equation.proof of necessary lengths of projections 10}
(V^*F^*f)(n')
=\ip{V^*F^*f}{\delta_{n'}}
=\ip{\Sigma^*U^*f}{\delta_{n'}}
=\ip{f}{U\Sigma\delta_{n'}}
=\left\{\begin{array}{cl}\sigma_{n'}\ip{f}{u_{n'}},&n'\leq M,\\0,&n'>M,\end{array}\right.
\end{equation}
where $\set{\sigma_{n'}}_{n'=1}^{\min\set{M,n}}$ are the singular values of $F$.  Since $(\Sigma^*\Sigma)(n',n')=\sigma_{n'}^2$ for any $n'=1,\dotsc,\min\set{M,n}$, 
\eqref{equation.proof of necessary lengths of projections 10} implies
\begin{equation}
\label{equation.proof of necessary lengths of projections 11}
\sum_{n'=1}^{n}\frac{\abs{(V^*F^*f)(n')}^2}{x-(\Sigma^*\Sigma)(n',n')}
=\sum_{n'=1}^{\min\set{M,n}}\frac{\sigma_{n'}^2\abs{\ip{f}{u_{n'}}}^2}{x-(\Sigma^*\Sigma)(n',n')}
=\sum_{n'=1}^{\min\set{M,n}}\frac{\sigma_{n'}^2}{x-\sigma_{n'}^2}\abs{\ip{f}{u_{n'}}}^2.
\end{equation}
Making the change of variables $m=n'$ in~\eqref{equation.proof of necessary lengths of projections 11} and substituting the result into~\eqref{equation.proof of necessary lengths of projections 9} gives
\begin{equation}
\label{equation.proof of necessary lengths of projections 12}
\tilde{q}(x)
=\tilde{p}(x)\,\biggparen{x-\|f\|^2-\sum_{m=1}^{\min\set{M,n}}\frac{\sigma_m^2}{x-\sigma_m^2}\abs{\ip{f}{u_m}}^2}\qquad \forall x\neq \sigma_1^2,\dotsc,\sigma_{\min\set{M,n}}^2,0.
\end{equation}
Here, the restriction that \smash{$x\neq \sigma_1^2,\dotsc,\sigma_{\min\set{M,n}}^2,0$} follows from the previously stated assumption that $x$ is not equal to any diagonal entry of $\Sigma^*\Sigma$; the set of these entries is $\set{\sigma_{n'}^2}_{n'=1}^{n}$ if $M\geq n$ and is $\set{\sigma_{n'}^2}_{n'=1}^{M}\cup\set{0}_{n'=1}^{n}$ if $M<n$.  Now recall that $p(x)$ and $q(x)$ are the $M$th degree characteristic polynomials of $FF^*$ and $GG^*$, respectively, while $\tilde{p}(x)$ is the $n$th degree characteristic polynomial of $F^*F$ and $\tilde{q}(x)$ is the $(n+1)$st degree characteristic polynomial of $G^*G$.  We now consider these facts along with \eqref{equation.proof of necessary lengths of projections 12} in two distinct cases: $n<M$ and $M\leq n$.  In the case where $n<M$, we have that $p(x)=x^{M-n}\tilde{p}(x)$ and $q(x)=x^{M-n-1}\tilde{q}(x)$.  Moreover, in this case the eigenvalues $\set{\beta_m}_{m=1}^{M}$ of $FF^*=U\Sigma\Sigma^*U^*$ are given by $\beta_m=\sigma_m^2$ for all $m=1,\dotsc,n$ and $\beta_m=0$ for all $m=n+1,\dotsc,M$, implying~\eqref{equation.proof of necessary lengths of projections 12} becomes
\begin{align}
\nonumber
\frac{q(x)}{x^{M-n-1}}
&=\frac{p(x)}{x^{M-n}}\,\biggparen{x-\|f\|^2-\sum_{m=1}^{n}\frac{\beta_m}{x-\beta_m}\abs{\ip{f}{u_m}}^2}\\
\label{equation.proof of necessary lengths of projections 13}
&=\frac{p(x)}{x^{M-n}}\,\biggparen{x-\|f\|^2-\sum_{m=1}^{M}\frac{\beta_m}{x-\beta_m}\abs{\ip{f}{u_m}}^2}\qquad \forall x\neq \beta_1,\dotsc,\beta_M,0.
\end{align}
In the remaining case where $M\leq n$, we have $\tilde{p}(x)=x^{n-M}p(x)$, $\tilde{q}(x)=x^{n+1-M}q(x)$ and $\beta_m=\sigma_m^2$ for all $m=1,\dotsc,M$, implying~\eqref{equation.proof of necessary lengths of projections 12} becomes
\begin{equation}
\label{equation.proof of necessary lengths of projections 14}
x^{n+1-M}q(x)
=x^{n-M}p(x)\,\biggparen{x-\|f\|^2-\sum_{m=1}^{M}\frac{\beta_m}{x-\beta_m}\abs{\ip{f}{u_m}}^2}\qquad \forall x\neq \beta_1,\dotsc,\beta_M,0.
\end{equation}
We now note that~\eqref{equation.proof of necessary lengths of projections 13} and~\eqref{equation.proof of necessary lengths of projections 14} are equivalent.   That is, regardless of the relationship between $M$ and $n$, we have
\begin{equation*}
\frac{q(x)}{p(x)}
=\frac1{x}\,\biggparen{x-\|f\|^2-\sum_{m=1}^{M}\frac{\beta_m}{x-\beta_m}\abs{\ip{f}{u_m}}^2}\qquad \forall x\neq \beta_1,\dotsc,\beta_M,0.
\end{equation*}
Writing $\displaystyle\norm{f}^2=\sum_{m=1}^{M}\abs{\ip{f}{u_m}}^2$ and then grouping the eigenvalues $\Lambda=\set{\beta_m}_{m=1}^{M}$ according to multiplicity gives
\begin{equation*}
\frac{q(x)}{p(x)}
=\frac1x\,\biggparen{x-\sum_{m=1}^{M}\abs{\ip{f}{u_m}}^2-\sum_{m=1}^{M}\frac{\beta_m}{x-\beta_m}\abs{\ip{f}{u_m}}^2}
=1-\sum_{m=1}^{M}\frac{\abs{\ip{f}{u_m}}^2}{x-\beta_m}
=1-\sum_{\lambda'\in\Lambda}\frac{\norm{P_{\lambda'}f}^2}{x-\lambda'}\qquad \forall x\notin\Lambda\cup\set{0}.
\end{equation*}
As such, for any $\lambda\in\Lambda$,
\begin{equation*}
\lim_{x\rightarrow\lambda}(x-\lambda)\frac{q(x)}{p(x)}
=\lim_{x\rightarrow\lambda}(x-\lambda)\biggparen{1-\sum_{\lambda'\in\Lambda}\frac{\norm{P_{\lambda'}f}^2}{x-\lambda'}}
=\lim_{x\rightarrow\lambda}\biggbracket{(x-\lambda)-\norm{P_{\lambda}f}^2-\sum_{\lambda'\neq\lambda}\norm{P_{\lambda'}f}^2\frac{x-\lambda}{x-\lambda'}}
=-\norm{P_{\lambda}f}^2
\end{equation*}
yielding our claim.
\end{proof}

Though technical, the proofs of the next two lemmas are nonetheless elementary, depending only on basic algebra and calculus.  As such, these proofs are given in the appendix.
\begin{lemma}
\label{lemma.interlacing and nonpositive limits}
If $\set{\beta_m}_{m=1}^{M}$ and $\set{\gamma_m}_{m=1}^{M}$ are real and nonincreasing, then $\set{\beta_m}_{m=1}^{M}\sqsubseteq\set{\gamma_m}_{m=1}^{M}$ if and only if
\begin{equation*}
\lim_{x\rightarrow\beta_m}(x-\beta_m)\frac{q(x)}{p(x)}\leq0\qquad\forall m=1,\dotsc,M,
\end{equation*}
where $\displaystyle p(x)=\prod_{m=1}^{M}(x-\beta_m)$ and $\displaystyle q(x)=\prod_{m=1}^{M}(x-\gamma_m)$.
\end{lemma}

\begin{lemma}
\label{lemma.equal limits}
If $\set{\beta_m}_{m=1}^{M}$, $\set{\gamma_m}_{m=1}^{M}$, and $\set{\delta_m}_{m=1}^{M}$ are real and nonincreasing and
\begin{equation*}
\lim_{x\rightarrow\beta_m}(x-\beta_m)\frac{q(x)}{p(x)}=\lim_{x\rightarrow\beta_m}(x-\beta_m)\frac{r(x)}{p(x)}\qquad\forall m=1,\dotsc,M,
\end{equation*}
where $\displaystyle p(x)=\prod_{m=1}^{M}(x-\beta_m)$, $\displaystyle q(x)=\prod_{m=1}^{M}(x-\gamma_m)$ and $\displaystyle r(x)=\prod_{m=1}^{M}(x-\delta_m)$, then $q(x)=r(x)$.
\end{lemma}

With Theorem~\ref{theorem.necessary lengths of projections} and Lemmas~\ref{lemma.eigensteps yield desired properties},~\ref{lemma.interlacing and nonpositive limits}~and~\ref{lemma.equal limits} in hand, we are ready to prove the main result of this section.
\begin{proof}[Proof of Theorem~\ref{theorem.necessity and sufficiency of eigensteps}]
($\Rightarrow$) Let $\set{\lambda_m}_{m=1}^{M}$ and $\set{\mu_n}_{n=1}^{N}$ be arbitrary nonnegative nonincreasing sequences, and let $F=\set{f_n}_{n=1}^{N}$ be any sequence of vectors such that the spectrum of $F^{}F^*$ is $\set{\lambda_m}_{m=1}^{M}$ and $\norm{f_n}^2=\mu_n$ for all $n=1,\dotsc,N$.   We claim that this particular $F$ can be constructed by following Steps~A and~B.

In particular, consider the sequence of sequences $\set{\set{\lambda_{n;m}}_{m=1}^{M}}_{n=0}^{N}$ defined by letting $\set{\lambda_{n;m}}_{m=1}^{M}$ be the spectrum of the frame operator $F_n^{}F_n^*$ of the sequence $F_n=\set{f_{n'}}_{n'=1}^{n}$ for all $n=1,\dotsc,N$ and letting $\lambda_{0;m}=0$ for all $m$.  We claim that $\set{\set{\lambda_{n;m}}_{m=1}^{M}}_{n=0}^{N}$ satisfies Definition~\ref{definition.eigensteps} and therefore is a valid sequence of eigensteps.  Note conditions (i) and (ii) of Definition~\ref{definition.eigensteps} are immediately satisfied.  To see that $\set{\set{\lambda_{n;m}}_{m=1}^{M}}_{n=0}^{N}$ satisfies (iii), consider the polynomials $p_n(x)$ defined by \eqref{equation.necessity and sufficiency of eigensteps 1} for all $n=1,\dotsc,N$.  In the special case where $n=1$, the desired property (iii) that $\set{0}_{m=1}^{M}\sqsubseteq\set{\lambda_{1;m}}_{m=1}^{M}$ follows from the fact that the spectrum $\set{\lambda_{1;m}}_{m=1}^{M}$ of the scaled rank-one projection $F_1^{}F_1^*=f_1^{}f_1^*$ is the value $\norm{f_1}^2=\mu_1$ along with $M-1$ repetitions of $0$, the eigenspaces being the span of $f_1$ and its orthogonal complement, respectively.  Meanwhile if $n=2,\dotsc,N$, Theorem~\ref{theorem.necessary lengths of projections} gives that
\begin{equation*}
\lim_{x\rightarrow\lambda_{n-1;m}}(x-\lambda_{n-1;m})\frac{p_n(x)}{p_{n-1}(x)}
=-\norm{P_{n-1;\lambda_{n-1;m}}f_n}^2
\leq0
\qquad \forall m=1,\dotsc,M,
\end{equation*}
implying by Lemma~\ref{lemma.interlacing and nonpositive limits} that $\set{\lambda_{n-1;m}}_{m=1}^{M}\sqsubseteq\set{\lambda_{n;m}}_{m=1}^{M}$ as claimed.  Finally, (iv) holds since for any $n=1,\dotsc,N$ we have
\begin{equation*}
\sum_{m=1}^{M}\lambda_{n;m}
=\Tr(F_n^{}F_n^*)
=\Tr(F_n^*F_n^{})
=\sum_{n'=1}^{n}\norm{f_{n'}}^2
=\sum_{n'=1}^{n}\mu_{n'}.
\end{equation*}

Having shown that these particular values of $\set{\set{\lambda_{n;m}}_{m=1}^{M}}_{n=0}^{N}$ can indeed be chosen in Step~A, we next show that our particular $F$ can be constructed according to Step~B.  As the method of Step~B is iterative, we use induction to prove that it can yield $F$.  Indeed, the only restriction that Step~B places on $f_1$ is that $\norm{f_1}^2=\mu_1$, something our particular $f_1$ satisfies by assumption.  Now assume that for any $n=1,\dotsc,N-1$ we have already correctly produced $\set{f_{n'}}_{n'=1}^{n}$ by following the method of Step~B; we show that we can produce the correct $f_{n+1}$ by continuing to follow Step~B.  To be clear, each iteration of Step~B does not produce a unique vector, but rather presents a family of $f_{n+1}$'s to choose from, and we show that our particular choice of $f_{n+1}$ lies in this family.  Specifically, our choice of $f_{n+1}$ must satisfy~\eqref{equation.necessity and sufficiency of eigensteps 2} for any choice of $\lambda\in\set{\lambda_{n;m}}_{m=1}^{M}$; the fact that it indeed does so follows immediately from Theorem~\ref{theorem.necessary lengths of projections}.  To summarize, we have shown that by making appropriate choices, we can indeed produce our particular $F$ by following Steps~A and~B, concluding this direction of the proof.

($\Leftarrow$)  Now assume that a sequence of vectors $F=\set{f_n}_{n=1}^{N}$ has been produced according to Steps~A and~B.  To be precise, letting $\set{\set{\lambda_{n;m}}_{m=1}^{M}}_{n=0}^{N}$ be the sequence of eigensteps chosen in Step~A, we claim that any $F=\set{f_n}_{n=1}^{N}$ constructed according to Step~B has the property that the spectrum of the frame operator $F_n^{}F_n^*$ of $F_n=\set{f_{n'}}_{n'=1}^{n}$ is $\set{\lambda_{n;m}}_{m=1}^{M}$ for all $n=1,\dotsc,N$.  Note that by Lemma~\ref{lemma.eigensteps yield desired properties}, proving this claim will yield our stated result that the spectrum of $FF^*$ is $\set{\lambda_m}_{m=1}^{M}$ and that $\norm{f_n}^2=\mu_n$ for all $n=1,\dotsc,N$.  As the method of Step~B is iterative, we prove this claim by induction.  Step~B begins by taking any $f_1$ such that $\norm{f_1}^2=\mu_1$.  As noted above in the proof of the other direction, the spectrum of $F_1^{}F_1^*=f_1^{}f_1^*$ is the value $\mu_1$ along with $M-1$ repetitions of $0$.  As claimed, these values match those of $\set{\lambda_{1;m}}_{m=1}^{M}$; to see this, note that Definition~\ref{definition.eigensteps}(i) and (iii) give $\set{0}_{m=1}^{M}=\set{\lambda_{0;m}}_{m=1}^{M}\sqsubseteq\set{\lambda_{1;m}}_{m=1}^{M}$ and so $\lambda_{1;m}=0$ for all $m=2,\dotsc,M$, at which point Definition~\ref{definition.eigensteps}(iv) implies $\lambda_{1,1}=\mu_1$.

Now assume that for any $n=1,\dotsc,N-1$, the Step~B process has already produced $F_n=\set{f_{n'}}_{n'=1}^{n}$ such that the spectrum of $F_n^{}F_n^*$ is $\set{\lambda_{n;m}}_{m=1}^{M}$.  We show that by following Step~B, we produce an $f_{n+1}$ such that $F_{n+1}=\set{f_{n'}}_{n'=1}^{n+1}$ has the property that $\set{\lambda_{n+1;m}}_{m=1}^{M}$ is the spectrum of $F_{n+1}^{}F_{n+1}^*$.  To do this, consider the polynomials $p_n(x)$ and $p_{n+1}(x)$ defined by~\eqref{equation.necessity and sufficiency of eigensteps 1} and pick any $f_{n+1}$ that satisfies~\eqref{equation.necessity and sufficiency of eigensteps 2}, namely
\begin{equation}
\label{equation.proof of necessity and sufficiency of eigensteps 1}
\lim_{x\rightarrow\lambda_{n;m}}(x-\lambda_{n;m})\frac{p_{n+1}(x)}{p_n(x)}
=-\norm{P_{n;\lambda_{n;m}}f_{n+1}}^2\qquad\forall m=1,\dotsc,M.
\end{equation}
Letting $\set{\hat{\lambda}_{n+1;m}}_{m=1}^{M}$ denote the spectrum of $F_{n+1}^{}F_{n+1}^*$, our goal is to show that $\set{\hat{\lambda}_{n+1;m}}_{m=1}^{M}=\set{\lambda_{n+1;m}}_{m=1}^{M}$.  Equivalently, our goal is to show that $p_{n+1}(x)=\hat{p}_{n+1}(x)$ where $\hat{p}_{n+1}(x)$ is the polynomial
\begin{equation*}
\hat{p}_{n+1}(x):=\prod_{m=1}^{M}(x-\hat{\lambda}_{n+1;m}).
\end{equation*}
Since $p_n(x)$ and $\hat{p}_{n+1}(x)$ are the characteristic polynomials of $F_n^{}F_n^*$ and $F_{n+1}^{}F_{n+1}^*$, respectively, Theorem~\ref{theorem.necessary lengths of projections} gives:
\begin{equation}
\label{equation.proof of necessity and sufficiency of eigensteps 2}
\lim_{x\rightarrow\lambda_{n;m}}(x-\lambda_{n;m})\frac{\hat{p}_{n+1}(x)}{p_n(x)}
=-\norm{P_{n;\lambda_{n;m}}f_{n+1}}^2\qquad\forall m=1,\dotsc,M.
\end{equation}
Comparing~\eqref{equation.proof of necessity and sufficiency of eigensteps 1} and~\eqref{equation.proof of necessity and sufficiency of eigensteps 2} gives:
\begin{equation*}
\lim_{x\rightarrow\lambda_{n;m}}(x-\lambda_{n;m})\frac{p_{n+1}(x)}{p_n(x)}
=\lim_{x\rightarrow\lambda_{n;m}}(x-\lambda_{n;m})\frac{\hat{p}_{n+1}(x)}{p_n(x)}\qquad\forall m=1,\dotsc,M,
\end{equation*}
implying by Lemma~\ref{lemma.equal limits} that $p_{n+1}(x)=\hat{p}_{n+1}(x)$, as desired.
\end{proof}


\section{Constructing frame elements from eigensteps}

As discussed in the previous section, Theorem~\ref{theorem.necessity and sufficiency of eigensteps} provides a two-step process for constructing any and all sequences of vectors $F=\set{f_n}_{n=1}^{N}$ in $\bbH_M$ whose frame operator possesses a given spectrum $\set{\lambda_m}_{m=1}^{M}$ and whose vectors have given lengths $\set{\mu_n}_{n=1}^{N}$.  In Step~A, we choose a sequence of eigensteps $\set{\set{\lambda_{n;m}}_{m=1}^{M}}_{n=0}^{N}$.  In the end, the $n$th sequence $\set{\lambda_{n;m}}_{m=1}^{M}$ will become the spectrum of the $n$th partial frame operator $F_n^{}F_n^*$, where $F_n=\set{f_{n'}}_{n'=1}^n$.  Due to the complexity of Definition~\ref{definition.eigensteps}, it is not obvious how to sequentially pick such eigensteps.   Looking at simple examples of this problem, such as the one discussed in Example~\ref{example.5 in 3} below, it appears as though the proof techniques needed to address these questions are completely different from those used throughout this paper.  As such, we leave the problem of parametrizing the eigensteps themselves for future work.  In this section, we thus focus on refining Step~B.

To be precise, the purpose of Step~B is to explicitly construct any and all sequences of vectors whose partial-frame-operator spectra match the eigensteps chosen in Step~A.  The problem with Step~B of Theorem~\ref{theorem.necessity and sufficiency of eigensteps}  is that it is not very explicit.  Indeed for every $n=1,\dotsc,N-1$, in order to construct $f_{n+1}$ we must first compute an orthonormal eigenbasis for $F_n^{}F_n^{*}$.  This problem is readily doable since the eigenvalues $\set{\lambda_{n;m}}_{m=1}^{M}$ of $F_n^{}F_n^*$ are already known.  It is nevertheless a tedious and inelegant process to do by hand, requiring us to, for example, compute QR-factorizations of $\lambda_{n;m}\rmI-F_n^{}F_n^*$ for each $m=1,\dotsc,M$.  This section is devoted to the following result, which is a version of Theorem~\ref{theorem.necessity and sufficiency of eigensteps} equipped with a more explicit Step~B; though technical, this new and improved Step~B is still simple enough to be performed by hand, a fact which will hopefully permit its future application to both theoretical and numerical problems.

\begin{theorem}
\label{theorem.explicit frame construction}
For any nonnegative nonincreasing sequences $\set{\lambda_m}_{m=1}^{M}$ and $\set{\mu_n}_{n=1}^{N}$, every sequence of vectors $F=\set{f_n}_{n=1}^{N}$ in $\bbH_M$ whose frame operator $FF^*$ has spectrum $\set{\lambda_m}_{m=1}^{M}$ and which satisfies $\norm{f_n}^2=\mu_n$ for all $n$ can be constructed by the following algorithm:
\begin{enumerate}
\renewcommand{\labelenumi}{\Alph{enumi}.}
\item
Pick eigensteps $\set{\set{\lambda_{n;m}}_{m=1}^{M}}_{n=0}^{N}$ as in Definition~\ref{definition.eigensteps}.\smallskip
\item
Let $U_1$ be any unitary matrix, $U_1=\{u_{1;m}\}_{m=1}^M$, and let $f_1=\sqrt{\mu_1} u_{1;1}$.  For each $n=1,\dots,N-1$:\smallskip
\begin{enumerate}
\renewcommand{\labelenumii}{\Alph{enumi}.\arabic{enumii}} 
\item Let $V_n$ be an $M \times M$ block-diagonal unitary matrix whose blocks correspond to the distinct values of $\{\lambda_{n;m}\}_{m=1}^M$ with the size of each block being the multiplicity of the corresponding eigenvalue.\smallskip
\item Identify those terms which are common to both $\{\lambda_{n;m}\}_{m=1}^M$ and $\{ \lambda_{n+1;m}\}_{m=1}^M$.  Specifically: \smallskip
\begin{itemize}
\item
Let $\calI_n\subseteq\{1,\dots,M\}$ consist of those indices $m$ such that $\lambda_{n;m}<\lambda_{n;m'}$ for all $m'<m$ and such that the multiplicity of $\lambda_{n;m}$ as a value in $\{\lambda_{n;m'}\}_{m'=1}^M$ exceeds its multiplicity as a value in $\{\lambda_{n+1;m'}\}_{m'=1}^M$.
\item 
Let $\calJ_n\subseteq\{1,\dots,M\}$ consist of those indices $m$ such that $\lambda_{n+1;m}<\lambda_{n+1;m'}$ for all $m'<m$ and such that the multiplicity of $\lambda_{n;m}$ as a value in $\{\lambda_{n+1;m'}\}_{m'=1}^M$ exceeds its multiplicity as a value in $\{\lambda_{n;m'}\}_{m'=1}^M$.
\end{itemize}
The sets $\calI_n$ and $\calJ_n$ have equal cardinality, which we denote $R_n$.  Next:
\begin{itemize}
\item Let $\pi_{\calI_n}$ be the unique permutation on $\{1,\dots,M\}$ that is increasing on both $\calI_n$ and $\calI_n^\rmc$ and such that $\pi_{\calI_n}(m)\in\{1,\dots,R_n\}$ for all $m\in\calI_n$.  Let $\Pi_{\calI_n}$ be the associated permutation matrix $\Pi_{\calI_n}\delta_{m}=\delta_{\pi_{\calI_n}(m)}$.
\item Let $\pi_{\calJ_n}$ be the unique permutation on $\{1,\dots,M\}$ that is increasing on both $\calJ_n$ and $\calJ_n^\rmc$ and such that $\pi_{\calJ_n}(m)\in\{1,\dots,R_n\}$ for all $m\in\calJ_n$.  Let $\Pi_{\calJ_n}$ be the associated permutation matrix $\Pi_{\calJ_n}\delta_{m}=\delta_{\pi_{\calJ_n}(m)}$.
\end{itemize}
\item Let $v_n$, $w_n$ be the $R_n\times 1$ vectors whose entries are
\begin{equation*} 
v_n(\pi_{\calI_n}(m))=\begin{bmatrix}-\,\frac{\displaystyle\prod_{m''\in\calJ_n} (\lambda_{n;m}-\lambda_{n+1,m''})}{\displaystyle\prod_{\substack{m''\in\calI_n\\m''\neq m}} (\lambda_{n;m}-\lambda_{n;m''})} \end{bmatrix}^{\frac12}\!\!,\,
w_n(\pi_{\calJ_n}(m'))=\begin{bmatrix} \frac{\displaystyle\prod_{m''\in\calI_n} (\lambda_{n+1;m'}-\lambda_{n,m''})}{\displaystyle\prod_{\substack{m''\in\calJ_n\\m''\neq m'}} (\lambda_{n+1;m'}-\lambda_{n+1;m''})}\end{bmatrix}^{\frac12}\quad \forall m\in\calI_n, m'\in\calJ_n.
\end{equation*}
\item$f_{n+1}=U_nV_n\Pi_{\calI_n}^{\rmT}\begin{bmatrix}v_n\\0\end{bmatrix}$, where the $M\times 1$ vector $\begin{bmatrix}v_n\\0\end{bmatrix}$ is $v_n$ padded with $M-R_n$ zeros.
\item $U_{n+1}=U_nV_n\Pi_{\calI_n}^{\rmT}\begin{bmatrix}W_n&0\\0&\rmI\end{bmatrix}\Pi_{\calJ_n}$ where $W_n$ is the $R_n\times R_n$ matrix whose entries are:
\begin{equation*}
W_n(\pi_{\calI_n}(m),\pi_{\calJ_n}(m')) = \frac1{\lambda_{n+1;m'}-\lambda_{n;m}} v_n(\pi_{\calI_n}(m))w_n(\pi_{\calJ_n}(m')).
\end{equation*}
\end{enumerate}
\end{enumerate}
Conversely, any $F$ constructed by this process has $\set{\lambda_m}_{m=1}^{M}$ as the spectrum of $FF^*$ and $\norm{f_n}^2=\mu_n$ for all $n$.

\noindent
Moreover, for any $F$ constructed in this manner and any $n=1,\dotsc,N$, the spectrum of the frame operator $F_n^{}F_n^*$ arising from the partial sequence $F_n=\set{f_{n'}}_{n'=1}^{n}$ is $\set{\lambda_{n;m}}_{m=1}^{M}$, and the columns of $U_n$ form a corresponding orthonormal eigenbasis for $F_n^{}F_n^*$.
\end{theorem}

Before proving Theorem~\ref{theorem.explicit frame construction}, we give an example of its implementation, with the hope of conveying the simplicity of the underlying idea, and better explaining the heavy notation used in the statement of the result.
\begin{example}
\label{example.5 in 3}
We now use Theorem~\ref{theorem.explicit frame construction} to construct UNTFs consisting of $5$ vectors in $\bbR^3$.  Here, $\lambda_1=\lambda_2=\lambda_3=\frac53$ and $\mu_1=\mu_2=\mu_3=\mu_4=\mu_5=1$.  By Step~A, our first task is to pick a sequence of eigensteps consistent with Definition~\ref{definition.eigensteps}, that is, pick $\set{\lambda_{1;1},\lambda_{1;2},\lambda_{1;3}}$, $\set{\lambda_{2;1},\lambda_{2;2},\lambda_{2;3}}$, $\set{\lambda_{3;1},\lambda_{3;2},\lambda_{3;3}}$ and $\set{\lambda_{4;1},\lambda_{4;2},\lambda_{4;3}}$ that satisfy the interlacing conditions:
\begin{equation}
\label{equation.5 in 3 example 1}
\set{0,0,0}
\sqsubseteq\set{\lambda_{1;1},\lambda_{1;2},\lambda_{1;3}}
\sqsubseteq\set{\lambda_{2;1},\lambda_{2;2},\lambda_{2;3}}
\sqsubseteq\set{\lambda_{3;1},\lambda_{3;2},\lambda_{3;3}}
\sqsubseteq\set{\lambda_{4;1},\lambda_{4;2},\lambda_{4;3}}
\sqsubseteq\set{\tfrac53,\tfrac53;\tfrac53},
\end{equation}
as well as the trace conditions:
\begin{equation}
\label{equation.5 in 3 example 2}
\lambda_{1;1}+\lambda_{1;2}+\lambda_{1;3}=1,\qquad
\lambda_{2;1}+\lambda_{2;2}+\lambda_{2;3}=2,\qquad
\lambda_{3;1}+\lambda_{3;2}+\lambda_{3;3}=3,\qquad
\lambda_{4;1}+\lambda_{4;2}+\lambda_{4;3}=4.
\end{equation}
Writing these desired spectra in a table:
\begin{equation*} 
\begin{tabular}{p{1cm} p{1cm} p{1cm} p{1cm} p{1cm} p{1cm} l}
\ \,$n$&$0$&$1$&$2$&$3$&$4$&$5${\smallskip}\\ 
\hline\noalign{\smallskip}
$\lambda_{n;3}$&0 &? &? &? &? &$\frac{5}3${\smallskip}\\
$\lambda_{n;2}$&0 &? &? &? &? &$\frac{5}3${\smallskip}\\
$\lambda_{n;1}$&0 &? &? &? &? &$\frac{5}3$\\
\end{tabular}
\end{equation*}
the trace condition~\eqref{equation.5 in 3 example 2} means that the sum of the values in the $n$th column is $\sum_{n'=1}^n\mu_{n'}=n$, while the interlacing condition~\eqref{equation.5 in 3 example 1} means that any value $\lambda_{n;m}$ is at least the neighbor to the upper right $\lambda_{n+1;m+1}$ and no more than its neighbor to the right $\lambda_{n+1;m}$.  In particular, for $n=1$,  we necessarily have $0=\lambda_{0;2}\leq\lambda_{1;2}\leq\lambda_{0;1}=0$ and $0=\lambda_{0;3}\leq\lambda_{1;3}\leq\lambda_{0;2}=0$ implying that $\lambda_{1;2}=\lambda_{1;3}=0$.  Similarly, for $n=4$, interlacing requires that $\frac53=\lambda_{5;2}\leq\lambda_{4;1}\leq\lambda_{5;1}=\frac53$ and $\frac53=\lambda_{5;3}\leq\lambda_{4;2}\leq\lambda_{5;2}=\frac53$ implying that $\lambda_{4;1}=\lambda_{4;2}=\frac{5}3$.  That is, we necessarily have:
\begin{equation*} 
\begin{tabular}{p{1cm} p{1cm} p{1cm} p{1cm} p{1cm} p{1cm} l}
\ \,$n$&$0$&$1$&$2$&$3$&$4$&$5${\smallskip}\\ 
\hline\noalign{\smallskip}
$\lambda_{n;3}$&0 &0 &? &? &? &$\frac{5}3${\smallskip}\\
$\lambda_{n;2}$&0 &0 &? &? &$\frac{5}3$ &$\frac{5}3${\smallskip}\\
$\lambda_{n;1}$&0 &? &? &? &$\frac{5}3$ &$\frac{5}3$\\
\end{tabular}
\end{equation*}
Applying this same idea again for $n=2$ and $n=3$ gives $0=\lambda_{1;3}\leq\lambda_{2;3}\leq\lambda_{1;2}=0$ and $\frac53=\lambda_{4;2}\leq\lambda_{3;1}\leq\lambda_{4;1}=\frac53$, and so we also necessarily have that $\lambda_{2;3}=0$, and $\lambda_{3;1}=\frac{5}3$:
\begin{equation*} 
\begin{tabular}{p{1cm} p{1cm} p{1cm} p{1cm} p{1cm} p{1cm} l}
\ \,$n$&$0$&$1$&$2$&$3$&$4$&$5${\smallskip}\\ 
\hline\noalign{\smallskip}
$\lambda_{n;3}$&0 &0 &0 &? &? &$\frac{5}3${\smallskip}\\
$\lambda_{n;2}$&0 &0 &? &? &$\frac{5}3$ &$\frac{5}3${\smallskip}\\
$\lambda_{n;1}$&0 &? &? &$\frac{5}3$ &$\frac{5}3$ &$\frac{5}3$\\
\end{tabular}
\end{equation*}
Moreover, the trace condition~\eqref{equation.5 in 3 example 2} at $n=1$ gives $1=\lambda_{1;1}+\lambda_{1;2}+\lambda_{1;3}=\lambda_{1;1}+0+0$ and so $\lambda_{1;1}=1$.  Similarly, the trace condition at $n=4$ gives $4=\lambda_{4;1}+\lambda_{4;2}+\lambda_{4;3}=\frac53+\frac53+\lambda_{4;3}$ and so $\lambda_{4;3}=\frac23$:
\begin{equation*} 
\begin{tabular}{p{1cm} p{1cm} p{1cm} p{1cm} p{1cm} p{1cm} l}
\ \,$n$&$0$&$1$&$2$&$3$&$4$&$5${\smallskip}\\ 
\hline\noalign{\smallskip}
$\lambda_{n;3}$&0 &0 &0 &? &$\frac23$ &$\frac{5}3${\smallskip}\\
$\lambda_{n;2}$&0 &0 &? &? &$\frac{5}3$ &$\frac{5}3${\smallskip}\\
$\lambda_{n;1}$&0 &1 &? &$\frac{5}3$ &$\frac{5}3$ &$\frac{5}3$\\
\end{tabular}
\end{equation*}
The remaining entries are not fixed.  In particular, we let $\lambda_{3;3}$ be some variable $x$ and note that by the trace condition, $3=\lambda_{3;1}+\lambda_{3;2}+\lambda_{3;3}=x+\lambda_{3;2}+\frac53$ and so $\lambda_{3;2}=\frac43-x$.  Similarly letting $\lambda_{2;2}=y$ gives $\lambda_{2;1}=2-y$: 
\begin{equation}
\label{equation.5 in 3 example 3}
\begin{tabular}{p{1cm} p{1cm} p{1cm} p{1cm} p{1cm} p{1cm} l}
\ \,$n$&$0$&$1$&$2$&$3$&$4$&$5${\smallskip}\\ 
\hline\noalign{\smallskip}
$\lambda_{n;3}$&0 &0 &0 &$x$ &$\frac23$ &$\frac{5}3${\smallskip}\\
$\lambda_{n;2}$&0 &0 &$y$ &$\frac{4}3-x$ &$\frac{5}3$ &$\frac{5}3${\smallskip}\\
$\lambda_{n;1}$&0 &1 &$2-y$ &$\frac{5}3$ &$\frac{5}3$ &$\frac{5}3$\\
\end{tabular}
\end{equation}
We take care to note that $x$ and $y$ in~\eqref{equation.5 in 3 example 3} are not arbitrary, but instead must be chosen so that the interlacing relations~\eqref{equation.5 in 3 example 3} are satisfied.   In particular, we have:
\begin{align}
\nonumber
\set{\lambda_{3;1},\lambda_{3;2},\lambda_{3;3}}\sqsubseteq\set{\lambda_{4;1},\lambda_{4;2},\lambda_{4;3}}&\quad\Longleftrightarrow\quad x\leq\tfrac23\leq\tfrac43-x\leq\tfrac53,\\
\label{equation.5 in 3 example 4}
\set{\lambda_{2;1},\lambda_{2;2},\lambda_{2;3}}\sqsubseteq\set{\lambda_{3;1},\lambda_{3;2},\lambda_{3;3}}&\quad\Longleftrightarrow\quad 0\leq x\leq y\leq\tfrac43-x\leq2-y\leq\tfrac53,\\
\nonumber
\set{\lambda_{1;1},\lambda_{1;2},\lambda_{1;3}}\sqsubseteq\set{\lambda_{2;1},\lambda_{2;2},\lambda_{2;3}}&\quad\Longleftrightarrow\quad 0\leq y\leq 1\leq 2-y.
\end{align}
By plotting each of the $11$ inequalities of \eqref{equation.5 in 3 example 4} as a half-plane (Figure~\ref{figure.5 in 3}(a)), we obtain a $5$-sided convex set (Figure~\ref{figure.5 in 3}(b)) of all $(x,y)$ such that~\eqref{equation.5 in 3 example 3} is a valid sequence of eigensteps.
\begin{figure}
\begin{center}
\includegraphics{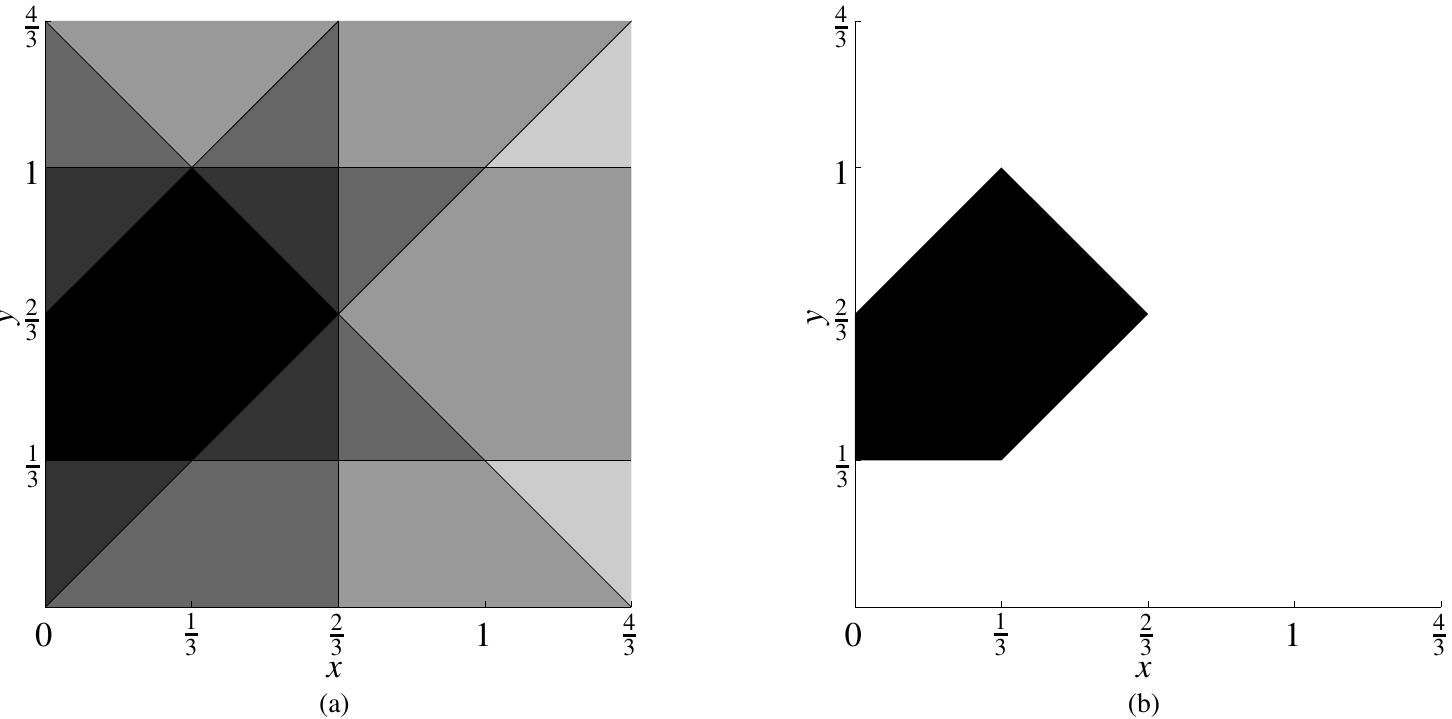}
\caption{\label{figure.5 in 3}Pairs of parameters $(x,y)$ that generate a valid sequence of eigensteps when substituted into~\eqref{equation.5 in 3 example 3}.  To be precise, in order to satisfy the interlacing requirements of Definition~\ref{definition.eigensteps}, $x$ and $y$ must be chosen so as to satisfy the $11$ pairwise inequalities summarized in~\eqref{equation.5 in 3 example 4}.  Each of these inequalities corresponds to a half-plane (a), and the set of $(x,y)$ that satisfy all of them is given by their intersection (b).  By Theorem~\ref{theorem.explicit frame construction}, any corresponding sequence of eigensteps~\eqref{equation.5 in 3 example 3} generates a $3\times 5$ UNTF and conversely, every $3\times 5$ UNTF is generated in this way.  As such, $x$ and $y$ may be viewed as the two essential parameters in the set of all such frames.  In particular, for $(x,y)$ that do not lie on the boundary of the set in (b), applying the algorithm of Theorem~\ref{theorem.explicit frame construction} to~\eqref{equation.5 in 3 example 3} and choosing $U_1=V_1=V_2=V_3=V_4=\rmI$ yields the $3\times 5$ UNTF whose elements are given in Table~\ref{table.5 in 3 parametrization}.
}
\end{center}
\end{figure}
Specifically, this set is the convex hull of $(0,\frac13)$, $(\frac13,\frac13)$, $(\frac23,\frac23)$, $(\frac13,1)$ and $(0,\frac23)$.  We note that though this analysis is straightforward in this case, it does not easily generalize to other cases in which $M$ and $N$ are large.

To complete Step~A of Theorem~\ref{theorem.explicit frame construction}, we pick any particular $(x,y)$ from the set depicted in Figure~\ref{figure.5 in 3}(b).  For example, if we pick $(x,y)=(0,\frac13)$ then~\eqref{equation.5 in 3 example 3} becomes:
\begin{equation}
\label{equation.5 in 3 example 5}
\begin{tabular}{p{1cm} p{1cm} p{1cm} p{1cm} p{1cm} p{1cm} l}
\ \,$n$&$0$&$1$&$2$&$3$&$4$&$5${\smallskip}\\ 
\hline\noalign{\smallskip}
$\lambda_{n;3}$&0 &0 &0 &$0$ &$\frac23$ &$\frac{5}3${\smallskip}\\
$\lambda_{n;2}$&0 &0 &$\frac13$ &$\frac{4}3$ &$\frac{5}3$ &$\frac{5}3${\smallskip}\\
$\lambda_{n;1}$&0 &1 &$\frac{5}3$ &$\frac{5}3$ &$\frac{5}3$ &$\frac{5}3$\\\end{tabular}
\end{equation}
We now perform Step~B of Theorem~\ref{theorem.explicit frame construction} for this particular choice of eigensteps.  First, we must choose a unitary matrix $U_1$.  Considering the equation for $U_{n+1}$ along with the fact that the columns of $U_N$ will form an eigenbasis for $F$, we see that our choice for $U_1$ merely rotates this eigenbasis, and hence the entire frame $F$, to our liking.  We choose $U_1=\rmI$ for the sake of simplicity.  Thus, 
\begin{equation*}
f_1=\sqrt{\mu_1}u_{1;1}=\begin{bmatrix}1\\0\\0\end{bmatrix}.
\end{equation*}
We now iterate, performing Steps~B.1 through~B.5 for $n=1$ to find $f_2$ and $U_2$, then performing Steps~B.1 through~B.5 for $n=2$ to find $f_3$ and $U_3$, and so on.  Throughout this process, the only remaining choices to be made appear in Step~B.1.  In particular, for $n=1$ Step~B.1 asks us to pick a block-diagonal unitary matrix $V_1$ whose blocks are sized according to the multiplicities of the eigenvalues $\set{\lambda_{1;1},\lambda_{1;2},\lambda_{1;3}}=\set{1,0,0}$.  That is, $V_1$ consists of a $1\times 1$ unitary block---a unimodular scalar---and a $2\times 2$ unitary block.  There are an infinite number of such $V_1$'s, each leading to a distinct frame.  For the sake of simplicity, we choose $V_1=\rmI$.  Having completed Step~B.1 for $n=1$, we turn to Step~B.2, which requires us to consider the columns of \eqref{equation.5 in 3 example 5} that correspond to $n=1$ and $n=2$:
\begin{equation}
\label{equation.5 in 3 example 6}
\begin{tabular}{p{1cm} p{1cm} l}
\ \,$n$&$1$&$2${\smallskip}\\ 
\hline\noalign{\smallskip}
$\lambda_{n;3}$&0 &0 {\smallskip}\\
$\lambda_{n;2}$&0 &$\frac13${\smallskip}\\
$\lambda_{n;1}$&1 &$\frac{5}3$
\end{tabular}
\end{equation}
In particular, we compute a set of indices $\calI_1\subseteq\set{1,2,3}$ that contains the indices $m$ of $\set{\lambda_{1;1},\lambda_{1;2},\lambda_{1;3}}=\set{1,0,0}$ for which (i) the multiplicity of $\lambda_{1;m}$ as a value of $\set{1,0,0}$ exceeds its multiplicity as a value of $\set{\lambda_{2;1},\lambda_{2;2},\lambda_{2;3}}=\set{\frac53,\frac13,0}$ and (ii) $m$ corresponds to the first occurrence of $\lambda_{1;m}$ as a value of $\set{1,0,0}$; by these criteria, we find $\calI_1=\set{1,2}$.  Similarly $m\in\calJ_1$ if and only if $m$ indicates the first occurrence of a value $\lambda_{2;m}$ whose multiplicity as a value of $\set{\frac53,\frac13,0}$ exceeds its multiplicity as a value of $\set{1,0,0}$, and so $\calJ_1=\set{1,2}$.  Equivalently, $\calI_1$ and $\calJ_1$ can be obtained by canceling common terms from~\eqref{equation.5 in 3 example 6}, working top to bottom; an explicit algorithm for doing so is given in Table~\ref{table.index set algorithm}.

Continuing with Step~B.2 for $n=1$, we now find the unique permutation $\pi_{\calI_1}:\set{1,2,3}\rightarrow\set{1,2,3}$ that is increasing on both $\calI_1=\set{1,2}$ and its complement $\calI_1^\rmc=\set3$ and takes $\calI_1$ to the first $R_1=\abs{\calI_1}=2$ elements of $\set{1,2,3}$.  In this particular instance, $\pi_{\calI_1}$ happens to be the identity permutation, and so $\Pi_{\calI_1}=\rmI$.  Since $\calJ_1=\set{1,2}=\calI_1$, we similarly have that $\pi_{\calJ_1}$ and $\Pi_{\calJ_1}$ are the identity permutation and matrix, respectively.

For the remaining steps, it is useful to isolate the terms in \eqref{equation.5 in 3 example 6} that correspond to $\calI_1$ and $\calJ_1$:
\begin{equation}
\label{equation.5 in 3 example 7}
\begin{tabular}{ll}
$\beta_2=\lambda_{1;2}=0$,\ &\ $\gamma_2=\lambda_{2;2}=\frac13$,\smallskip\\
$\beta_1=\lambda_{1;1}=1$,\ &\ $\gamma_1=\lambda_{2;1}=\frac{5}3$.
\end{tabular}
\end{equation}
In particular, in Step~B.3, we find the $R_1\times1=2\times1$ vector $v_1$ by computing quotients of products of differences of the values in~\eqref{equation.5 in 3 example 7}:
\begin{align}
\label{equation.5 in 3 example 8}
[v_1(1)]^2
&=-\frac{(\beta_1-\gamma_1)(\beta_1-\gamma_2)}{(\beta_1-\beta_2)}
=-\frac{(1-\frac53)(1-\frac13)}{(1-0)}
=\tfrac49,\\
\label{equation.5 in 3 example 9}
[v_1(2)]^2
&=-\frac{(\beta_2-\gamma_1)(\beta_2-\gamma_2)}{(\beta_2-\beta_1)}
=-\frac{(0-\frac53)(0-\frac13)}{(0-1)}
=\tfrac59,
\end{align}  
yielding $v_1=\begin{bmatrix}\frac23\smallskip\\\frac{\sqrt5}3\end{bmatrix}$.  Similarly, we compute $w_1=\begin{bmatrix}\frac{\sqrt5}{\sqrt6}\smallskip\\\frac1{\sqrt6}\end{bmatrix}$ according to the formulas:
\begin{align}
\label{equation.5 in 3 example 10}
[w_1(1)]^2
&=\frac{(\gamma_1-\beta_1)(\gamma_1-\beta_2)}{(\gamma_1-\gamma_2)}
=\frac{(\frac53-1)(\frac53-0)}{(\frac53-\frac13)}
=\tfrac56,\\
\label{equation.5 in 3 example 11}
[w_1(2)]^2
&=\frac{(\gamma_2-\beta_1)(\gamma_2-\beta_2)}{(\gamma_2-\gamma_1)}
=\frac{(\frac13-1)(\frac13-0)}{(\frac13-\frac53)}
=\tfrac16.
\end{align}  
Next, in Step~B.4, we form our second frame element $f_2=U_1V_1\Pi_{\calI_1}^\rmT\begin{bmatrix}v_1\\0\end{bmatrix}$:
\begin{equation*}
f_2
=\begin{bmatrix}1&0&0\\0&1&0\\0&0&1\end{bmatrix}\begin{bmatrix}1&0&0\\0&1&0\\0&0&1\end{bmatrix}\begin{bmatrix}1&0&0\\0&1&0\\0&0&1\end{bmatrix}\begin{bmatrix}\frac23\smallskip\\\frac{\sqrt5}3\smallskip\\0\end{bmatrix}
=\begin{bmatrix}\frac23\smallskip\\\frac{\sqrt5}3\smallskip\\0\end{bmatrix}.
\end{equation*}
As justified in the proof of Theorem~\ref{theorem.explicit frame construction}, the resulting partial sequence of vectors
\begin{equation*}
F_2
=\begin{bmatrix}f_1&f_2\end{bmatrix} 
=\begin{bmatrix}1&\frac23\smallskip\\0&\frac{\sqrt5}3\smallskip\\0&0\end{bmatrix}
\end{equation*}
has a frame operator $F_2^{}F_2^*$ whose spectrum is $\set{\lambda_{2;1},\lambda_{2;2},\lambda_{2;3}}=\set{\frac53,\frac13,0}$.  Moreover, a corresponding orthonormal eigenbasis for $F_2^{}F_2^*$ is computed in Step~B.5; here the first step is to compute the $R_1\times R_1=2\times 2$ matrix $W_1$ by computing a pointwise product of a certain $2\times 2$ matrix with the outer product of $v_1$ with $w_1$:
\begin{equation*}
W_1
=\begin{bmatrix}\frac1{\gamma_1-\beta_1}&\frac1{\gamma_2-\beta_1}\smallskip\\\frac1{\gamma_1-\beta_2}&\frac1{\gamma_2-\beta_2}\end{bmatrix}\odot\begin{bmatrix}v_1(1)\\v_1(2)\end{bmatrix}\begin{bmatrix}w_1(1)&w_1(2)\end{bmatrix}\\
=\begin{bmatrix}\frac32&-\frac32\smallskip\\\frac35&3\end{bmatrix}\odot\begin{bmatrix}\frac{2\sqrt5}{3\sqrt6}&\frac2{3\sqrt6}\smallskip\\\frac5{3\sqrt6}&\frac{\sqrt5}{3\sqrt6}\end{bmatrix}\\
=\begin{bmatrix}\frac{\sqrt5}{\sqrt6}&-\frac1{\sqrt6}\smallskip\\\frac1{\sqrt6}&\frac{\sqrt5}{\sqrt6}\end{bmatrix}.
\end{equation*}
Note that $W_1$ is a real orthogonal matrix whose diagonal and subdiagonal entries are strictly positive and whose superdiagonal entries are strictly negative; one can easily verify that every $W_n$ has this form.  More significantly, the proof of Theorem~\ref{theorem.explicit frame construction} guarantees that the columns of
\begin{equation*}
U_2
=U_1V_1\Pi_{\calI_1}^{\rmT}\begin{bmatrix}W_1&0\\0&\rmI\end{bmatrix}\Pi_{\calJ_1}
=\begin{bmatrix}1&0&0\\0&1&0\\0&0&1\end{bmatrix}\begin{bmatrix}1&0&0\\0&1&0\\0&0&1\end{bmatrix}\begin{bmatrix}1&0&0\\0&1&0\\0&0&1\end{bmatrix}\begin{bmatrix}\frac{\sqrt5}{\sqrt6}&-\frac1{\sqrt6}&0\smallskip\\\frac1{\sqrt6}&\frac{\sqrt5}{\sqrt6}&0\smallskip\\0&0&1\end{bmatrix}\begin{bmatrix}1&0&0\\0&1&0\\0&0&1\end{bmatrix}
=\begin{bmatrix}\frac{\sqrt5}{\sqrt6}&-\frac1{\sqrt6}&0\smallskip\\\frac1{\sqrt6}&\frac{\sqrt5}{\sqrt6}&0\smallskip\\0&0&1\end{bmatrix}
\end{equation*}
form an orthonormal eigenbasis of $F_2^{}F_2^*$.  This completes the $n=1$ iteration of Step~B; we now repeat this process for $n=2,3,4$.  For $n=2$, in Step~B.1 we arbitrarily pick some $3\times3$ diagonal unitary matrix $V_2$.  Note that if we wish our frame to be real, there are only $2^3=8$ such choices of $V_2$.  For the sake of simplicity, we choose $V_2=\rmI$ in this example.  Continuing, Step~B.2 involves canceling the common terms in
\begin{equation*}
\begin{tabular}{p{1cm} p{1cm} l}
\ \,$n$&$2$&$3${\smallskip}\\ 
\hline\noalign{\smallskip}
$\lambda_{n;3}$&$0$ 		&$0$ {\smallskip}\\
$\lambda_{n;2}$&$\frac13$	&$\frac43${\smallskip}\\
$\lambda_{n;1}$&$\frac53$	&$\frac53$
\end{tabular}
\end{equation*}
to find $\calI_2=\calJ_2=\set2$, and so 
\begin{equation*}
\Pi_{\calI_2}=\Pi_{\calJ_2}
=\begin{bmatrix}0&1&0\\1&0&0\\0&0&1\end{bmatrix}.
\end{equation*}  
In Step~B.3, we find that $v_2=w_2=\begin{bmatrix}1\end{bmatrix}$.  Steps~B.4 and B.5 then give that $F_3=\begin{bmatrix}f_1&f_2&f_3\end{bmatrix}$ and $U_3$ are
\begin{equation*}
F_3=\begin{bmatrix}1&\frac23&-\frac1{\sqrt6}\smallskip\\0&\frac{\sqrt5}3&\frac{\sqrt5}{\sqrt6}\smallskip\\0&0&0\end{bmatrix},\qquad
U_3=\begin{bmatrix}\frac{\sqrt5}{\sqrt6}&-\frac1{\sqrt6}&0\smallskip\\\frac1{\sqrt6}&\frac{\sqrt5}{\sqrt6}&0\smallskip\\0&0&1\end{bmatrix}.
\end{equation*}
The columns of $U_3$ form an orthonormal eigenbasis for the partial frame operator $F_3^{}F_3^*$ with corresponding eigenvalues $\set{\lambda_{3;1},\lambda_{3;2},\lambda_{3;3}}=\set{\frac53,\frac43,0}$.  For the $n=3$ iteration, we pick $V_3=\rmI$ and cancel the common terms in
\begin{equation*}
\begin{tabular}{p{1cm} p{1cm} l}
\ \,$n$&$3$&$4${\smallskip}\\ 
\hline\noalign{\smallskip}
$\lambda_{n;3}$&$0$ 		&$\frac23${\smallskip}\\
$\lambda_{n;2}$&$\frac43$	&$\frac53${\smallskip}\\
$\lambda_{n;1}$&$\frac53$	&$\frac53$
\end{tabular}
\end{equation*}
to obtain $\calI_3=\set{2,3}$ and $\calJ_3=\set{1,3}$, implying
\begin{equation*}
\Pi_{\calI_3}=\begin{bmatrix}0&1&0\\0&0&1\\1&0&0\end{bmatrix},
\qquad\Pi_{\calJ_3}=\begin{bmatrix}1&0&0\\0&0&1\\0&1&0\end{bmatrix},
\qquad
\begin{tabular}{ll}
$\beta_2=\lambda_{3;3}=0$,\ 		&\ $\gamma_2=\lambda_{4;3}=\frac23$,\smallskip\\
$\beta_1=\lambda_{3;2}=\frac43$,\ 	&\ $\gamma_1=\lambda_{4;1}=\frac53$.
\end{tabular}
\end{equation*}
In Step~B.3, we then compute the $R_3\times 1=2\times 1$ vectors $v_3$ and $w_3$ in a manner analogous to \eqref{equation.5 in 3 example 8}, \eqref{equation.5 in 3 example 9}, \eqref{equation.5 in 3 example 10} and \eqref{equation.5 in 3 example 11}:
\begin{equation*}
v_3=\begin{bmatrix}\frac1{\sqrt6}\smallskip\\\frac{\sqrt5}{\sqrt6}\end{bmatrix},
\qquad w_3=\begin{bmatrix}\frac{\sqrt5}3\smallskip\\\frac23\end{bmatrix}.
\end{equation*}
Note that in Step~B.4, the role of permutation matrix $\Pi_{\calI_3}^{\rmT}$ is that it maps the entries of $v_3$ onto the $\calI_3$ indices, meaning that $v_4$ lies in the span of the corresponding eigenvectors $\set{u_{3;m}}_{m\in\calI_3}$:
\begin{equation*}
f_4
=U_3V_3\Pi_{\calI_3}^{\rmT}\begin{bmatrix}v_3\\0\end{bmatrix}
=\begin{bmatrix}\frac{\sqrt5}{\sqrt6}&-\frac1{\sqrt6}&0\smallskip\\\frac1{\sqrt6}&\frac{\sqrt5}{\sqrt6}&0\smallskip\\0&0&1\end{bmatrix}\begin{bmatrix}1&0&0\\0&1&0\\0&0&1\end{bmatrix}\begin{bmatrix}0&0&1\\1&0&0\\0&1&0\end{bmatrix}\begin{bmatrix}\frac1{\sqrt6}\smallskip\\\frac{\sqrt5}{\sqrt6}\smallskip\\0\end{bmatrix}
=\begin{bmatrix}\frac{\sqrt5}{\sqrt6}&-\frac1{\sqrt6}&0\smallskip\\\frac1{\sqrt6}&\frac{\sqrt5}{\sqrt6}&0\smallskip\\0&0&1\end{bmatrix}\begin{bmatrix}0\smallskip\\\frac1{\sqrt6}\smallskip\\\frac{\sqrt5}{\sqrt6}\end{bmatrix}
=\begin{bmatrix}-\frac16\smallskip\\\frac{\sqrt5}6\smallskip\\\frac{\sqrt5}{\sqrt6}\end{bmatrix}.
\end{equation*}
In a similar fashion, the purpose of the permutation matrices in Step~B.5 is to embed the entries of the $2\times 2$ matrix $W_3$ into the $\calI_3=\set{2,3}$ rows and $\calJ_3=\set{1,3}$ columns of a $3\times 3$ matrix:
\begin{align*}
U_4
=U_3V_3\Pi_{\calI_3}^{\rmT}\begin{bmatrix}W_3&0\\0&\rmI\end{bmatrix}\Pi_{\calJ_3}
&=\begin{bmatrix}\frac{\sqrt5}{\sqrt6}&-\frac1{\sqrt6}&0\smallskip\\\frac1{\sqrt6}&\frac{\sqrt5}{\sqrt6}&0\smallskip\\0&0&1\end{bmatrix}\begin{bmatrix}1&0&0\\0&1&0\\0&0&1\end{bmatrix}\begin{bmatrix}0&0&1\\1&0&0\\0&1&0\end{bmatrix}\begin{bmatrix}\frac{\sqrt5}{\sqrt6}&-\frac1{\sqrt6}&0\smallskip\\\frac1{\sqrt6}&\frac{\sqrt5}{\sqrt6}&0\smallskip\\0&0&1\end{bmatrix}\begin{bmatrix}1&0&0\\0&0&1\\0&1&0\end{bmatrix}\\
&=\begin{bmatrix}\frac{\sqrt5}{\sqrt6}&-\frac1{\sqrt6}&0\smallskip\\\frac1{\sqrt6}&\frac{\sqrt5}{\sqrt6}&0\smallskip\\0&0&1\end{bmatrix}\begin{bmatrix}0&1&0\smallskip\\\frac{\sqrt5}{\sqrt6}&0&-\frac1{\sqrt6}\smallskip\\\frac1{\sqrt6}&0&\frac{\sqrt5}{\sqrt6}\end{bmatrix}\\
&=\begin{bmatrix}-\frac{\sqrt5}6&\frac{\sqrt5}{\sqrt6}&\frac16\smallskip\\\frac56&\frac1{\sqrt6}&-\frac{\sqrt5}6\smallskip\\\frac1{\sqrt6}&0&\frac{\sqrt5}{\sqrt6}\end{bmatrix}.
\end{align*}    
For the last iteration $n=4$, we again choose $V_4=\rmI$ in Step~B.1.  For Step~B.2, note that since
\begin{equation*}
\begin{tabular}{p{1cm} p{1cm} l}
\ \,$n$&$4$&$5${\smallskip}\\ 
\hline\noalign{\smallskip}
$\lambda_{n;3}$&$\frac23$ 	&$\frac53${\smallskip}\\
$\lambda_{n;2}$&$\frac53$	&$\frac53${\smallskip}\\
$\lambda_{n;1}$&$\frac53$	&$\frac53$
\end{tabular}
\end{equation*}
we have $\calI_4=\set3$ and $\calJ_4=\set1$, implying
\begin{equation*}
\Pi_{\calI_4}=\begin{bmatrix}0&0&1\\1&0&0\\0&1&0\end{bmatrix},\qquad 
\Pi_{\calJ_4}=\begin{bmatrix}1&0&0\\0&1&0\\0&0&1\end{bmatrix}.
\end{equation*}
Working through Steps~B.3, B.4 and B.5 yields the UNTF:
\begin{equation}
\label{equation.5 in 3 example 12}
F=F_5=\begin{bmatrix}1&\frac23&-\frac1{\sqrt6} & -\frac16 &\frac16\smallskip\\0&\frac{\sqrt5}3&\frac{\sqrt5}{\sqrt6}&\frac{\sqrt5}6&-\frac{\sqrt5}6\smallskip\\0&0&0&\frac{\sqrt5}{\sqrt6}&\frac{\sqrt5}{\sqrt6}\end{bmatrix},\qquad
U_5=\begin{bmatrix}\frac16&-\frac{\sqrt5}6&\frac{\sqrt5}{\sqrt6}\smallskip\\-\frac{\sqrt5}{6}&\frac{5}6&\frac1{\sqrt6}\smallskip\\\frac{\sqrt5}{\sqrt6}&\frac1{\sqrt6} &0\end{bmatrix}.
\end{equation}
We emphasize that the UNTF $F$ given in \eqref{equation.5 in 3 example 12} was based on the particular choice of eigensteps given in~\eqref{equation.5 in 3 example 5}, which arose by choosing $(x,y)=(0,\frac13)$ in~\eqref{equation.5 in 3 example 3}.  Choosing other pairs $(x,y)$ from the parameter set depicted in Figure~\ref{figure.5 in 3}(b) yields other UNTFs.  Indeed, since the eigensteps of a given $F$ are equal to those of $UF$ for any unitary operator $U$, we have in fact that each distinct $(x,y)$ yields a UNTF which is not unitarily equivalent to any of the others.  For example, by following the algorithm of Theorem~\ref{theorem.explicit frame construction} and choosing $U_1=\rmI$ and $V_n=\rmI$ in each iteration, we obtain the following four additional UNTFs, each corresponding to a distinct corner point of the parameter set:
\newcommand\PBS[1]{\let\temp=\\#1\let\\=\temp}
\begin{align*}
F\raggedleft &=\left[ \begin{tabular}{r p{0.8cm} p{0.8cm} p{0.8cm} p{0.8cm} }
$1$ \raggedleft & $\frac23$ \raggedleft & $0$ \raggedleft & $-\frac13$ \raggedleft &  $-\frac13${\smallskip} \PBS\raggedleft\\
$0$ \raggedleft & $\frac{\sqrt5}3$ \raggedleft & $0$ \raggedleft & $\frac{\sqrt5}3$ \raggedleft &$\frac{\sqrt5}3${\smallskip}\PBS\raggedleft\\
$0$ \raggedleft & $0$ \raggedleft & $1$ \raggedleft & $\frac1{\sqrt3}$\raggedleft &$-\frac1{\sqrt3}${\smallskip}\PBS\raggedleft\\
\end{tabular} \right] \qquad \text{for $(x,y)=(\tfrac13,\tfrac13)$,} \PBS\raggedleft\\ 
F\raggedleft &=\left[ \begin{tabular}{r p{0.8cm} p{0.8cm} p{0.8cm} p{0.8cm} }
$1$ \raggedleft & $\frac13$ \raggedleft & $\frac13$ \raggedleft & $-\frac13$ \raggedleft &$-\frac1{\sqrt3}${\smallskip}\PBS\raggedleft\\
$0$ \raggedleft & $\frac{\sqrt{8}}3$ \raggedleft & $\frac1{3 \sqrt2}$ \raggedleft & $-\frac1{3 \sqrt2}$ \raggedleft &$\frac{\sqrt2}{\sqrt3}${\smallskip}\PBS\raggedleft\\
$0$ \raggedleft & $0$ \raggedleft & $\frac{\sqrt5}{\sqrt6}$ \raggedleft & $\frac{\sqrt5}{\sqrt6}$ \raggedleft & $0$ {\smallskip}\PBS\raggedleft\\
\end{tabular} \right] \qquad \text{for $(x,y)=(\tfrac23,\tfrac23)$,} \PBS\raggedleft\\ 
F\raggedleft &=\left[ \begin{tabular}{r p{0.8cm} p{0.8cm} p{0.8cm} p{0.8cm} }
$1$ \raggedleft & $0$ \raggedleft & $0$ \raggedleft &$\frac1{\sqrt3}$ \raggedleft & $-\frac1{\sqrt3}${\smallskip}\PBS\raggedleft\\
$0$ \raggedleft & $1$ \raggedleft & $\frac23$ \raggedleft & $-\frac13$ \raggedleft &$-\frac13${\smallskip}\PBS\raggedleft\\
$0$ \raggedleft & $0$ \raggedleft & $\frac{\sqrt5}3$ \raggedleft & $\frac{\sqrt5}3$ \raggedleft & $\frac{\sqrt5}3${\smallskip}\PBS\raggedleft\\
\end{tabular} \right] \qquad \text{for $(x,y)=(\tfrac13,1)$,} \PBS\raggedleft\\          
F\raggedleft &=\left[ \begin{tabular}{r p{0.8cm} p{0.8cm} p{0.8cm} p{0.8cm} }
$1$ \raggedleft & $\frac13$ \raggedleft & $-\frac1{\sqrt3}$ \raggedleft &$\frac13$ \raggedleft & $-\frac13${\smallskip}\PBS\raggedleft\\
$0$ \raggedleft & $\frac{\sqrt{8}}3$ \raggedleft & $\frac{\sqrt2}{\sqrt3}$ \raggedleft & $\frac1{3 \sqrt2}$ \raggedleft &$-\frac1{3 \sqrt2}${\smallskip}\PBS\raggedleft\\
$0$ \raggedleft & $0$ \raggedleft & $0$ \raggedleft & $\frac{\sqrt5}{\sqrt6}$ \raggedleft & $\frac{\sqrt5}{\sqrt6}${\smallskip}\PBS\raggedleft\\
\end{tabular} \right] \qquad \text{for $(x,y)=(0,\tfrac23)$.}                                                                
\end{align*}
Notice that, of the four UNTFs above, the second and fourth are actually the same up to a permutation of the frame elements.  This is an artifact of our method of construction, namely, that our choices for eigensteps, $U_1$, and $\set{V_n}_{n=1}^{N-1}$ determine the \textit{sequence} of frame elements.  As such, we can recover all permutations of a given frame by modifying these choices.  

We emphasize that these four UNTFs along with that of~\eqref{equation.5 in 3 example 12} are but five examples from the continuum of all such frames.  Indeed, keeping $x$ and $y$ as variables in~\eqref{equation.5 in 3 example 3} and applying the algorithm of Theorem~\ref{theorem.explicit frame construction}---again choosing $U_1=\rmI$ and $V_n=\rmI$ in each iteration for the sake of simplicity---yields the frame elements given in Table~\ref{table.5 in 3 parametrization}.  Here, we restrict $(x,y)$ so as to not lie on the boundary of the parameter set of Figure~\ref{figure.5 in 3}(b).  This restriction simplifies the analysis, as it prevents all unnecessary repetitions of values in neighboring columns in~\eqref{equation.5 in 3 example 3}.  Table~\ref{table.5 in 3 parametrization} gives an explicit parametrization for a two-dimensional manifold that lies within the set of all UNTFs consisting of five elements in three-dimensional space.  By Theorem~\ref{theorem.explicit frame construction}, this can be generalized so as to yield all such frames, provided we both (i) further consider $(x,y)$ that lie on each of the five line segments that constitute the boundary of the parameter set and (ii) throughout generalize $V_n$ to an arbitrary block-diagonal unitary matrix, where the sizes of the blocks are chosen in accordance with Step~B.1.
\end{example}

\begin{table}
\begin{small}
\begin{align*}
f_1&=\begin{bmatrix}1\\0\\0\end{bmatrix}\\
f_2&=\begin{bmatrix}1-y\\\sqrt{y(2-y)}\\0\end{bmatrix}\\
f_3&=\left[\begin{array}{ccc}%
\hspace{12.5pt}\frac{\sqrt{(3y-1)(2+3x-3y)(2-x-y)}}{6\sqrt{1-y}}\hspace{12.5pt}%
&-&\frac{\sqrt{(5-3y)(4-3x-3y)(y-x)}}{6\sqrt{1-y}}\smallskip\\%
\frac{\sqrt{y(3y-1)(2+3x-3y)(2-x-y)}}{6\sqrt{(1-y)(2-y)}}
&+&\frac{\sqrt{(5-3y)(2-y)(4-3x-3y)(y-x)}}{6\sqrt{y(1-y)}}\smallskip\\%
\frac{\sqrt{5x(4-3x)}}{3\sqrt{y(2-y)}}%
\end{array}\right]\\
f_4&=\left[\begin{array}{ccccccc}%
-\frac{\sqrt{(4-3x)(3y-1)(2-x-y)(4-3x-3y)}}{12\sqrt{(2-3x)(1-y)}}%
&-&\frac{\sqrt{(4-3x)(5-3y)(y-x)(2+3x-3y)}}{12\sqrt{(2-3x)(1-y)}}%
&-&\frac{\sqrt{x(3y-1)(y-x)(2+3x-3y)}}{4\sqrt{3(2-3x)(1-y)}}%
&+&\frac{\sqrt{x(5-3y)(2-x-y)(4-3x-3y)}}{4\sqrt{3(2-3x)(1-y)}}\smallskip\\%
-\frac{\sqrt{(4-3x)y(3y-1)(2-x-y)(4-3x-3y)}}{12\sqrt{(2-3x)(1-y)(2-y)}}%
&+&\frac{\sqrt{(4-3x)(2-y)(5-3y)(y-x)(2+3x-3y)}}{12\sqrt{(2-3x)y(1-y)}}%
&-&\frac{\sqrt{xy(3y-1)(y-x)(2+3x-3y)}}{4\sqrt{3(2-3x)(1-y)(2-y)}}%
&-&\frac{\sqrt{x(2-y)(5-3y)(2-x-y)(4-3x-3y)}}{4\sqrt{3(2-3x)y(1-y)}}\smallskip\\%
\frac{\sqrt{5x(2+3x-3y)(4-3x-3y)}}{6\sqrt{(2-3x)y(2-y)}}%
&+&\frac{\sqrt{5(4-3x)(y-x)(2-x-y)}}{2\sqrt{3(2-3x)y(2-y)}}%
\end{array}\right]\\
f_5&=\left[\begin{array}{ccccccc}%
\hspace{4.25pt}\frac{\sqrt{(4-3x)(3y-1)(2-x-y)(4-3x-3y)}}{12\sqrt{(2-3x)(1-y)}}\hspace{4.25pt}%
&+&\frac{\sqrt{(4-3x)(5-3y)(y-x)(2+3x-3y)}}{12\sqrt{(2-3x)(1-y)}}%
&-&\frac{\sqrt{x(3y-1)(y-x)(2+3x-3y)}}{4\sqrt{3(2-3x)(1-y)}}%
&+&\frac{\sqrt{x(5-3y)(2-x-y)(4-3x-3y)}}{4\sqrt{3(2-3x)(1-y)}}\smallskip\\%
\frac{\sqrt{(4-3x)y(3y-1)(2-x-y)(4-3x-3y)}}{12\sqrt{(2-3x)(1-y)(2-y)}}%
&-&\frac{\sqrt{(4-3x)(2-y)(5-3y)(y-x)(2+3x-3y)}}{12\sqrt{(2-3x)y(1-y)}}%
&-&\frac{\sqrt{xy(3y-1)(y-x)(2+3x-3y)}}{4\sqrt{3(2-3x)(1-y)(2-y)}}%
&-&\frac{\sqrt{x(2-y)(5-3y)(2-x-y)(4-3x-3y)}}{4\sqrt{3(2-3x)y(1-y)}}\smallskip\\%
-\frac{\sqrt{5x(2+3x-3y)(4-3x-3y)}}{6\sqrt{(2-3x)y(2-y)}}%
&+&\frac{\sqrt{5(4-3x)(y-x)(2-x-y)}}{2\sqrt{3(2-3x)y(2-y)}}%
\end{array}\right]
\end{align*}
\end{small}
\caption{\label{table.5 in 3 parametrization}  A continuum of UNTFs.  To be precise, for each choice of $(x,y)$ that lies in the interior of the parameter set depicted in Figure~\ref{figure.5 in 3}(b), these five elements form a UNTF for $\bbR^3$, meaning that its $3\times 5$ synthesis matrix $F$ has both unit norm columns and orthogonal rows of constant squared norm $\frac53$.  These frames were produced by applying the algorithm of Theorem~\ref{theorem.explicit frame construction} to the sequence of eigensteps given in~\eqref{equation.5 in 3 example 3}, choosing $U_1=\rmI$ and $V_n=\rmI$ for all $n$.  These formulas give an explicit parametrization for a $2$-dimensional manifold that lies within the set of all $3\times 5$ UNTFs.  By Theorem~\ref{theorem.explicit frame construction}, every such UNTF arises in this manner, with the understanding that $(x,y)$ may indeed be chosen from the boundary of the parameter set and that the initial eigenbasis $U_1$ and the block-diagonal unitary matrices $V_n$ are not necessarily the identity.} 
\end{table}

Having discussed the utility of Theorem~\ref{theorem.explicit frame construction}, we turn to its proof.

\begin{proof}[Proof of Theorem~\ref{theorem.explicit frame construction}]
($\Leftarrow$) Let $\set{\lambda_m}_{m=1}^{M}$ and $\set{\mu_n}_{n=1}^{N}$ be arbitrary nonnegative nonincreasing sequences and take an arbitrary sequence of eigensteps $\set{\set{\lambda_{n;m}}_{m=1}^{M}}_{n=0}^{N}$ in accordance with Definition~\ref{definition.eigensteps}.  Note here we do not assume that such a sequence of eigensteps actually exists for this particular choice of $\set{\lambda_m}_{m=1}^{M}$ and $\set{\mu_n}_{n=1}^{N}$; if one does not, then this direction of the result is vacuously true.

We claim that any $F=\set{f_n}_{n=1}^{N}$ constructed according to Step~B has the property that for all $n=1,\dotsc,N$, the spectrum of the frame operator $F_n^{}F_n^*$ of $F_n=\set{f_{n'}}_{n'=1}^{n}$ is $\set{\lambda_{n;m}}_{m=1}^{M}$, and that the columns of $U_n$ form an orthonormal eigenbasis for $F_n^{}F_n^*$.  Note that by Lemma~\ref{lemma.eigensteps yield desired properties}, proving this claim will yield our stated result that the spectrum of $FF^*$ is $\set{\lambda_m}_{m=1}^{M}$ and that $\norm{f_n}^2=\mu_n$ for all $n=1,\dotsc,N$.  Since Step~B is an iterative algorithm, we prove this claim by induction on $n$.  To be precise, Step~B begins by letting $U_1=\{u_{1;m}\}_{m=1}^M$ and $f_1=\sqrt{\mu_1} u_{1;1}$.  The columns of $U_1$ form an orthonormal eigenbasis for $F_1^{}F_1^*$ since $U_1$ is unitary by assumption and
\begin{equation*}
F_1^{}F_1^{*}u_{1;m}
=\ip{u_{1;m}}{f_1}f_1
=\ip{u_{1;m}}{\sqrt{\mu_1}u_{1;1}}\sqrt{\mu_1}u_{1;1}
=\mu_1\ip{u_{1;m}}{u_{1;1}}u_{1;1}
=\left\{\begin{array}{ll}\mu_1 u_{1;1}&m=1,\\0&m\neq 1,\end{array}\right.
\end{equation*}
for all $m=1,\dotsc,M$.  As such, the spectrum of $F_1^{}F_1^*$ consists of $\mu_1$ and $M-1$ repetitions of $0$.  To see that this spectrum matches the values of $\set{\lambda_{1;m}}_{m=1}^{M}$, note that by Definition~\ref{definition.eigensteps}, we know $\set{\lambda_{1;m}}_{m=1}^{M}$ interlaces on the trivial sequence $\set{\lambda_{0;m}}_{m=1}^{M}=\set{0}_{m=1}^{M}$ in the sense of~\eqref{equation.definition of interlacing}, implying $\lambda_{1;m}=0$ for all $m\geq2$; this in hand, note this definition further gives that $\lambda_{1;1}=\sum_{m=1}^{M}\lambda_{1;m}=\mu_1$.  Thus, our claim indeed holds for $n=1$.

We now proceed by induction, assuming that for any given $n=1,\dotsc,N-1$ the process of Step~B has produced $F_n=\set{f_{n'}}_{n'=1}^{n}$ such that the spectrum of $F_n^{}F_n^*$ is $\set{\lambda_{n;m}}_{m=1}^{M}$ and that the columns of $U_n$ form an orthonormal eigenbasis for $F_n^{}F_n^*$.  In particular, we have $F_n^{}F_n^*U_n^{}=U_n^{}D_n^{}$ where $D_n$ is the diagonal matrix whose diagonal entries are $\set{\lambda_{n;m}}_{m=1}^{M}$.  Defining $D_{n+1}$ analogously from $\set{\lambda_{n+1;m}}_{m=1}^{M}$, we show that constructing $f_{n+1}$ and $U_{n+1}$ according to Step~B implies $F_{n+1}^{}F_{n+1}^*U_{n+1}^{}=U_{n+1}^{}D_{n+1}^{}$ where $U_{n+1}$ is unitary; doing such proves our claim.

To do so, pick any unitary matrix $V_n$ according to Step~B.1.  To be precise, let $K_n$ denote the number of distinct values in $\{\lambda_{n;m}\}_{m=1}^M$, and for any $k=1,\dotsc,K_n$, let $L_{n;k}$ denote the multiplicity of the $k$th value.  We write the index $m$ as an increasing function of $k$ and $l$, that is, we write $\smash\{\lambda_{n;m}\}_{m=1}^M$ as \smash{$\{\lambda_{n;m(k,l)}\}_{k=1}^{K_n}\,_{l=1}^{L_{\smash{n;k}}}$} where $m(k,l)<m(k',l')$ if $k<k'$ or if $k=k'$ and $l<l'$.  We let $V_n$ be an $M\times M$ block-diagonal unitary matrix consisting of $K$ diagonal blocks, where for any $k=1,\dotsc,K$, the $k$th block is an $L_{n;k}\times L_{n;k}$ unitary matrix.  In the extreme case where all the values of $\{\lambda_{n;m}\}_{m=1}^M$ are distinct, we have that $V_n$ is a diagonal unitary matrix, meaning it is a diagonal matrix whose diagonal entries are unimodular.  Even in this case, there is some freedom in how to choose $V_n$; this is the only freedom that the Step~B process provides when determining $f_{n+1}$.  In any case, the crucial fact about $V_n$ is that its blocks match those corresponding to distinct multiples of the identity that appear along the diagonal of $D_n$, implying $D_nV_n=V_nD_n$.

Having chosen $V_n$, we proceed to Step~B.2.  Here, we produce subsets $\calI_n$ and $\calJ_n$ of $\set{1,\dotsc,M}$ that are the remnants of the indices of $\{\lambda_{n;m}\}_{m=1}^M$ and $\{\lambda_{n+1;m}\}_{m=1}^M$, respectively, obtained by canceling the values that are common to both sequences, working backwards from index $M$ to index $1$.  An explicit algorithm for doing so is given in Table~\ref{table.index set algorithm}.  
\begin{table}
\begin{center}
\begin{tabular}{ll}
01&$\calI_n^{(M)}:=\set{1,\dotsc,M}$\\
02&$\calJ_n^{(M)}:=\set{1,\dotsc,M}$\\
03&for $m=M,\dotsc,1$\\
04&\quad if $\lambda_{n;m}\in\{\lambda_{n+1;m'}\}_{m'\in\calJ_n^{(m)}}$\\
05&\quad\quad $\calI_{n}^{(m-1)}:=\calI_{n}^{(m)}\setminus\{m\}$\\
06&\quad\quad $\calJ_{n}^{(m-1)}:=\calJ_{n}^{(m)}\setminus \{m'\}$ where $m'=\max{\{m''\in\calJ_{n}^{(m)}:\lambda_{n+1;m''}=\lambda_{n;m}}\}$\\
07&\quad else\\
08&\quad\quad $\calI_{n}^{(m-1)}:=\calI_{n}^{(m)}$\\
09&\quad\quad $\calJ_{n}^{(m-1)}:=\calJ_{n}^{(m)}$\\
10&\quad end if\\
11&end for\\
12&$\calI_n:=\calI_n^{(1)}$\\
13&$\calJ_n:=\calI_n^{(1)}$\\
\end{tabular}
\end{center}
\caption{\label{table.index set algorithm}An explicit algorithm for computing the index sets $\calI_n$ and $\calJ_n$ in Step~B.2 of Theorem~\ref{theorem.explicit frame construction}}
\end{table}
Note that for each $m=M,\dotsc,1$ (Line~03), we either remove a single element from both $\calI_n^{(m)}$ and $\calJ_n^{(m)}$ (Lines~04--06) or remove nothing from both (Lines~07--09), meaning that $\calI_n:=\calI_n^{(1)}$ and $\calJ_n:=\calJ_n^{(1)}$ have the same cardinality, which we denote $R_n$.  Moreover, since $\set{\lambda_{n+1;m}}_{m=1}^{M}$ interlaces on $\set{\lambda_{n;m}}_{m=1}^{M}$, then for any real scalar $\lambda$ whose multiplicity as a value of $\set{\lambda_{n;m}}_{m=1}^{M}$ is $L$, we have that its multiplicity as a value of $\set{\lambda_{n+1;m}}_{m=1}^{M}$ is either $L-1$, $L$ or $L+1$.  When these two multiplicities are equal, this algorithm completely removes the corresponding indices from both $\calI_n$ and $\calJ_n$.  On the other hand, if the new multiplicity is $L-1$ or $L+1$, then the least such index in $\calI_n$ or $\calJ_n$ is left behind, respectively, leading to the definitions of $\calI_n$ or $\calJ_n$ given in Step~B.2.  Having these sets, it is trivial to find the corresponding permutations $\pi_{\calI_n}$ and $\pi_{\calJ_n}$ on $\set{1,\dotsc,M}$ and to construct the associated projection matrices $\Pi_{\calI_n}$ and $\Pi_{\calJ_n}$.

We now proceed to Step~B.3.  For the sake of notational simplicity, let $\set{\beta_r}_{r=1}^{R_n}$ and $\set{\gamma_r}_{r=1}^{R_n}$ denote the values of $\set{\lambda_{n;m}}_{m\in\calI_n}$ and $\set{\lambda_{n+1;m}}_{m\in\calJ_n}$, respectively.  That is, let $\beta_{\pi_{\calI_n}(m)}=\lambda_{n;m}$ for all $m\in\calI_n$ and $\gamma_{\pi_{\calJ_n}(m)}=\lambda_{n+1;m}$ for all $m\in\calJ_n$.  Note that due to the way in which $\calI_n$ and $\calJ_n$ were defined, we have that the values of $\set{\beta_r}_{r=1}^{R_n}$ and $\set{\gamma_r}_{r=1}^{R_n}$ are all distinct, both within each sequence and across the two sequences.  Moreover, since $\set{\lambda_{n;m}}_{m\in\calI_n}$ and $\set{\lambda_{n+1;m}}_{m\in\calJ_n}$ are nonincreasing while $\pi_{\calI_n}$ and $\pi_{\calJ_n}$ are increasing on $\calI_n$ and $\calJ_n$ respectively, then the values $\set{\beta_r}_{r=1}^{R_n}$ and $\set{\gamma_r}_{r=1}^{R_n}$ are strictly decreasing.  We further claim that $\set{\gamma_r}_{r=1}^{R_n}$ interlaces on $\set{\beta_r}_{r=1}^{R_n}$.  To see this, consider the four polynomials:
\begin{equation}
\label{equation.proof of explicit frame construction 1}
p_n(x)
=\prod_{m=1}^M(x-\lambda_{n;m}),\quad
p_{n+1}(x)
=\prod_{m=1}^M(x-\lambda_{n+1;m}),\quad
b(x)
=\prod_{r=1}^{R_n}(x-\beta_r),\quad
c(x)
=\prod_{r=1}^{R_n}(x-\gamma_r).
\end{equation}
Since $\set{\beta_r}_{r=1}^{R_n}$ and $\set{\gamma_r}_{r=1}^{R_n}$ were obtained by canceling the common terms from $\set{\lambda_{n;m}}_{m=1}^M$ and $\set{\lambda_{n+1;m}}_{m=1}^M$, we have that $p_{n+1}(x)/p_n(x)=c(x)/b(x)$ for all $x\notin\set{\lambda_{n;m}}_{m=1}^M$.  Writing any $r=1,\dotsc,R_n$ as $r=\pi_{\calI_n}(m)$ for some $m\in\calI_n$, we have that since $\set{\lambda_{n;m}}_{m=1}^M\sqsubseteq\set{\lambda_{n+1;m}}_{m=1}^M$, applying the ``only if" direction of Lemma~\ref{lemma.interlacing and nonpositive limits} with ``$p(x)$" and ``$q(x)$" being $p_n(x)$ and $p_{n+1}(x)$ gives
\begin{equation}
\label{equation.proof of explicit frame construction 2}
\lim_{x\rightarrow\beta_r}(x-\beta_r)\frac{c(x)}{b(x)}
=\lim_{x\rightarrow\lambda_{n;m}}(x-\lambda_{n;m})\frac{p_{n+1}(x)}{p_n(x)}
\leq0.
\end{equation}
Since~\eqref{equation.proof of explicit frame construction 2} holds for all $r=1,\dotsc,R_n$, applying ``if" direction of Lemma~\ref{lemma.interlacing and nonpositive limits} with ``$p(x)$" and ``$q(x)$" being $b(x)$ and $c(x)$ gives that $\set{\gamma_r}_{r=1}^{R_n}$ indeed interlaces on $\set{\beta_r}_{r=1}^{R_n}$.

Taken together, the facts that $\set{\beta_r}_{r=1}^{R_n}$ and $\set{\gamma_r}_{r=1}^{R_n}$ are distinct, strictly decreasing and interlacing sequences implies that the $R_n\times 1$ vectors $v_n$ and $w_n$ are well-defined.  To be precise, Step~B.3 may be rewritten as finding $v_n(r),w_n(r')\geq0$ for all $r,r'=1\dotsc,R_n$ such that
\begin{equation}
\label{equation.proof of explicit frame construction 3}
[v_n(r)]^2=-\,\frac{\displaystyle\prod_{r''=1}^{R_n} (\beta_r-\gamma_{r''})}{\displaystyle\prod_{\substack{r''=1\\r''\neq r}}^{R}(\beta_r-\beta_{r''})},
\qquad
[w_n(r')]^2=\frac{\displaystyle\prod_{r''=1}^{R_n} (\gamma_{r'}-\beta_{r''})}{\displaystyle\prod_{\substack{r''=1\\r''\neq r'}}^{R}(\gamma_{r'}-\gamma_{r''})}.
\end{equation}
Note the fact that the $\beta_r$'s and $\gamma_r$'s are distinct implies that the denominators in \eqref{equation.proof of explicit frame construction 3} are nonzero, and moreover that the quotients themselves are nonzero.  In fact, since $\set{\beta_r}_{r=1}^{R_n}$ is strictly decreasing, then for any fixed $r$, the values \smash{$\set{\beta_r-\beta_{r''}}_{r''\neq r}$} can be decomposed into $r-1$ negative values \smash{$\set{\beta_r-\beta_{r''}}_{r''=1}^{r-1}$} and $R_n-r$ positive values \smash{$\set{\beta_r-\beta_{r''}}_{r''=r+1}^{R_n}$}.  Moreover, since $\set{\beta_r}_{r=1}^{R_n}\sqsubseteq\set{\gamma_r}_{r=1}^{R_n}$, then for any such $r$, the values $\set{\beta_r-\gamma_{r''}}_{r''=1}^{R_n}$ can be broken into $r$ negative values \smash{$\set{\beta_r-\gamma_{r''}}_{r''=1}^{r}$} and $R_n-r$ positive values \smash{$\set{\beta_r-\gamma_{r''}}_{r''=r+1}^{R_n}$}.  With the inclusion of an additional negative sign, we see that the quantity defining $[v_n(r)]^2$ in~\eqref{equation.proof of explicit frame construction 3} is indeed positive.  Meanwhile, the quantity defining $[w_n(r')]^2$ has exactly $r'-1$ negative values in both the numerator and denominator, namely \smash{$\set{\gamma_{r'}-\beta_{r''}}_{r''=1}^{r'-1}$} and \smash{$\set{\gamma_{r'}-\gamma_{r''}}_{r''=1}^{r'-1}$}, respectively.

Having shown that the $v_n$ and $w_n$ of Step~B.3 are well-defined, we now take $f_{n+1}$ and $U_{n+1}$ as defined in Steps~B.4 and B.5.  Recall that what remains to be shown in this direction of the proof is that $U_{n+1}$ is a unitary matrix and that $F_{n+1}=\set{f_{n'}}_{n'=1}^{n+1}$ satisfies $F_{n+1}^{}F_{n+1}^{*}U_{n+1}^{}=U_{n+1}^{}D_{n+1}^{}$.  To do so, consider the definition of $U_{n+1}$ and recall that $U_n$ is unitary by the inductive hypothesis, $V_n$ is unitary by construction, and that the permutation matrices $\Pi_{\calI_n}$ and $\Pi_{\calJ_n}$ are orthogonal, that is, unitary and real.  As such, to show that $U_{n+1}$ is unitary, it suffices to show that the $R_n\times R_n$ real matrix $W_n$ is orthogonal.  To do this, recall that eigenvectors corresponding to distinct eigenvalues of self-adjoint operators are necessarily orthogonal.  As such, to show that $W_n$ is orthogonal, it suffices to show that the columns of $W_n$ are eigenvectors of a real symmetric operator.  To this end, we claim
\begin{equation}
\label{equation.proof of explicit frame construction 4}
(D_{n;\calI_n}+v_n^{}v_n^\rmT)W_n
=W_nD_{n+1;\calJ_n},
\qquad
W_n^\rmT W_n^{}(r,r)=1,\ \forall r=1,\dotsc,R_n,
\end{equation}
where $D_{n;\calI_n}$ and $D_{n+1;\calJ_n}$ are the $R_n\times R_n$ diagonal matrices whose $r$th diagonal entries are given by \smash{$\beta_r=\lambda_{n;\pi_{\calI_n}^{-1}(r)}$} and \smash{$\gamma_r=\lambda_{n+1;\pi_{\calJ_n}^{-1}(r)}$}, respectively.  To prove \eqref{equation.proof of explicit frame construction 4}, note that for any $r,r'=1,\dotsc,R_n$, 
\begin{equation}
\label{equation.proof of explicit frame construction 5}
[(D_{n;\calI_n}+v_n^{}v_n^\rmT)W_n](r,r')
=(D_{n;\calI_n}W_n)(r,r')+(v_n^{}v_n^\rmT W_n)(r,r')
=\beta_r W_n(r,r')+v_n(r)\sum_{r''=1}^{R_n}v_n(r'')W_n(r'',r').
\end{equation}
Rewriting the definition of $W_n$ from Step~B.5 in terms of $\set{\beta_r}_{r=1}^{R_n}$ and $\set{\gamma_r}_{r=1}^{R_n}$ gives
\begin{equation}
\label{equation.proof of explicit frame construction 6}
W_n(r,r')=\frac{v_n(r)w_n(r')}{\gamma_{r'}-\beta_{r}}.
\end{equation}
Substituting~\eqref{equation.proof of explicit frame construction 6} into~\eqref{equation.proof of explicit frame construction 5} gives
\begin{align}
\nonumber
[(D_{n;\calI_n}+v_n^{}v_n^\rmT)W_n](r,r')
&=\beta_r\frac{v_n(r)w_n(r')}{\gamma_{r'}-\beta_{r}}+v_n(r)\sum_{r''=1}^{R_n}v_n(r'')\frac{v_n(r'')w_n(r')}{\gamma_{r'}-\beta_{r''}}\\
\label{equation.proof of explicit frame construction 7}
&=v_n(r)w_n(r')\,\biggparen{\frac{\beta_r}{\gamma_{r'}-\beta_{r}}+\sum_{r''=1}^{R_n}\frac{[v_n(r'')]^2}{\gamma_{r'}-\beta_{r''}}}\,.
\end{align}
Simplifying~\eqref{equation.proof of explicit frame construction 7} requires a polynomial identity.  Note that the difference $\smash\prod_{r''=1}^{R_n}(x-\gamma_{r''})-\prod_{r''=1}^{R_n}(x-\beta_{r''})$ of two monic polynomials is itself a polynomial of degree at most $R_n-1$, and as such it can be written as the Lagrange interpolating polynomial determined by the $R_n$ distinct points $\set{\beta_r}_{r=1}^{R_n}$:
\begin{equation}
\label{equation.proof of explicit frame construction 8}
\prod_{r''=1}^{R_n}(x-\gamma_{r''})-\prod_{r''=1}^{R_n}(x-\beta_{r''})
=\sum_{r''=1}^{R_n}\biggparen{\prod_{r=1}^{R_n}(\beta_{r''}-\gamma_{r})-0}\prod_{\substack{r=1\\r\neq r''}}^{R_n}\frac{(x-\beta_{r})}{(\beta_{r''}-\beta_{r})} 
=\sum_{r''=1}^{R_n}\frac{\displaystyle\prod_{r=1}^{R_n}(\beta_{r''}-\gamma_{r})}{\displaystyle\prod_{\substack{r=1\\r\neq r''}}^{R_n} (\beta_{r''}-\beta_{r})}\prod_{\substack{r=1\\r\neq r''}}^{R_n} (x-\beta_{r}).
\end{equation}
Recalling the expression for $[v_n(r)]^2$ given in~\eqref{equation.proof of explicit frame construction 3}, \eqref{equation.proof of explicit frame construction 8} can be rewritten as
\begin{equation}
\label{equation.proof of explicit frame construction 9}
\prod_{r''=1}^{R_n}(x-\beta_{r''})-\prod_{r''=1}^{R_n}(x-\gamma_{r''})
=\sum_{r''=1}^{R_n}[v_n(r'')]^2\displaystyle\prod_{\substack{r=1\\r\neq r''}}^{R_n} (x-\beta_{r}).
\end{equation}
Dividing both sides of \eqref{equation.proof of explicit frame construction 9} by $\displaystyle\prod_{r''=1}^{R_n}(x-\beta_{r''})$ gives
\begin{equation}
\label{equation.proof of explicit frame construction 10}
1-\prod_{r''=1}^{R_n}\frac{(x-\gamma_{r''})}{(x-\beta_{r''})}=\sum_{r''=1}^{R_n}\frac{[v_n(r'')]^2}{(x-\beta_{r''})}\qquad\forall x\notin\set{\beta_r}_{r=1}^{R_n}.
\end{equation}
For any $r'=1,\dotsc,R_n$, letting $x=\gamma_{r'}$ in \eqref{equation.proof of explicit frame construction 10} makes the left-hand product vanish, yielding the identity:
\begin{equation}
\label{equation.proof of explicit frame construction 11}
1=\sum_{r''=1}^{R_n}\frac{[v_n(r'')]^2}{(\gamma_{r'}-\beta_{r''})} \qquad \forall r'=1,\dotsc,R_n.
\end{equation}
Substituting~\eqref{equation.proof of explicit frame construction 11} into \eqref{equation.proof of explicit frame construction 7} and then recalling \eqref{equation.proof of explicit frame construction 6} gives
\begin{equation}
\label{equation.proof of explicit frame construction 12}
[(D_{n;\calI_n}+v_n^{}v_n^\rmT)W_n](r,r')
=v_n(r)w_n(r')\,\biggparen{\frac{\beta_r}{\gamma_{r'}-\beta_{r}}+1}
=\gamma_{r'}\frac{v_n(r)w_n(r')}{\gamma_{r'}-\beta_{r}}
=\gamma_{r'}W_n(r,r')
=(W_nD_{n+1;\calJ_n})(r,r').
\end{equation}
As~\eqref{equation.proof of explicit frame construction 12} holds for all $r,r'=1,\dotsc,R_n$ we have the first half of our claim~\eqref{equation.proof of explicit frame construction 4}.  In particular, we know that the columns of $W_n$ are eigenvectors of the real symmetric operator \smash{$D_{n;\calI_n}+v_n^{}v_n^\rmT$} which correspond to the distinct eigenvalues $\set{\gamma_r}_{r=1}^{R_n}$.  As such, the columns of $W_n$ are orthogonal.  To show that $W_n$ is an orthogonal matrix, we must further show that the columns of $W_n$ have unit norm, namely the second half of~\eqref{equation.proof of explicit frame construction 4}.  To prove this, at any $x\notin\set{\beta_r}_{r=1}^{R_n}$ we differentiate both sides of~\eqref{equation.proof of explicit frame construction 10} with respect to $x$ to obtain
\begin{equation}
\label{equation.proof of explicit frame construction 13}
\sum_{r''=1}^{R_n}\biggbracket{\prod_{\substack{r=1\\r\neq r''}}^{R_n} \frac{(x-\gamma_{r})}{(x-\beta_{r})}}\frac{\gamma_{r''}-\beta_{r''}}{(x-\beta_{r''})^2}
=\sum_{r''=1}^{R_n} \frac{[v_n(r'')]^2}{(x-\beta_{r''})^2}\qquad\forall x\notin\set{\beta_r}_{r=1}^{R_n}.
\end{equation}
For any $r'=1,\dotsc,R_n$, letting $x=\gamma_{r'}$ in \eqref{equation.proof of explicit frame construction 13} makes the left-hand summands where $r''\neq r'$ vanish; by \eqref{equation.proof of explicit frame construction 3}, the remaining summand where $r''=r'$ can be written as:
\begin{equation}
\label{equation.proof of explicit frame construction 14}
\frac1{[w_n(r')]^2}
=\frac{\displaystyle\prod_{\substack{r=1\\r\neq r'}}^{R}(\gamma_{r'}-\gamma_{r})}{\displaystyle\prod_{r=1}^{R_n} (\gamma_{r'}-\beta_{r})}
=\biggbracket{\prod_{\substack{r=1\\r\neq r'}}^{R_n} \frac{(\gamma_{r'}-\gamma_{r})}{(\gamma_{r'}-\beta_{r})}}\frac{\gamma_{r'}-\beta_{r'}}{(\gamma_{r'}-\beta_{r'})^2}
=\sum_{r''=1}^{R_n} \frac{[v_n(r'')]^2}{(\gamma_{r'}-\beta_{r''})^2}.
\end{equation}
We now use this identity to show that the columns of $W_n$ have unit norm; for any $r'=1,\dotsc,R_n$,~\eqref{equation.proof of explicit frame construction 6} and~\eqref{equation.proof of explicit frame construction 14} give
\begin{equation*}
(W_n^\rmT W_n^{})(r',r')
=\sum_{r''=1}^{R_n}[W_n(r'',r')]^2
=\sum_{r''=1}^{R_n}\biggparen{\frac{v_n(r'')w_n(r')}{\gamma_{r'}-\beta_{r''}}}^2
=[w_n(r')]^2\sum_{r''=1}^{R_n} \frac{[v_n(r'')]^2}{(\gamma_{r'}-\beta_{r''})^2}
=[w_n(r')]^2\frac1{[w_n(r')]^2}
=1.
\end{equation*}  
Having shown that $W_n$ is orthogonal, we have that $U_{n+1}$ is unitary.

For this direction of the proof, all that remains to be shown is that \smash{$F_{n+1}^{}F_{n+1}^{*}U_{n+1}^{}=U_{n+1}^{}D_{n+1}^{}$}.  To do this, we write \smash{$F_{n+1}^{}F_{n+1}^{*}=F_n^{}F_n^{*}+f_{n+1}^{}f_{n+1}^{*}$} and recall the definition of $U_{n+1}$:
\begin{align}
\nonumber
F_{n+1}^{}F_{n+1}^{*}U_{n+1}^{}
&=(F_{n}^{}F_{n}^{*}+f_{n+1}^{}f_{n+1}^{*})U_n^{}V_n^{}\Pi_{\calI_n^{}}^\rmT\begin{bmatrix}W_n&0\\0&\rmI\end{bmatrix}\Pi_{\calJ_n^{}}\\
\label{equation.proof of explicit frame construction 15}
&=F_{n}^{}F_{n}^{*}U_n^{}V_n^{}\Pi_{\calI_n^{}}^\rmT\begin{bmatrix}W_n&0\\0&\rmI\end{bmatrix}\Pi_{\calJ_n^{}}+f_{n+1}^{}f_{n+1}^{*}U_n^{}V_n^{}\Pi_{\calI_n^{}}^\rmT\begin{bmatrix}W_n&0\\0&\rmI\end{bmatrix}\Pi_{\calJ_n^{}}.
\end{align}
To simplify the first term in \eqref{equation.proof of explicit frame construction 15}, recall that the inductive hypothesis gives $F_n^{}F_n^{*}U_n^{}=U_n^{}D_n^{}$ and that $V_n$ was constructed to satisfy $D_n^{}V_n^{}=V_n^{}D_n^{}$, implying
\begin{equation}
\label{equation.proof of explicit frame construction 16}
F_{n}^{}F_{n}^{*}U_n^{}V_n^{}\Pi_{\calI_n^{}}^\rmT\begin{bmatrix}W_n&0\\0&\rmI\end{bmatrix}\Pi_{\calJ_n^{}}
=U_n^{}V_n^{}D_n^{}\Pi_{\calI_n^{}}^\rmT\begin{bmatrix}W_n&0\\0&\rmI\end{bmatrix}\Pi_{\calJ_n^{}}
=U_n^{}V_n^{}\Pi_{\calI_n^{}}^\rmT(\Pi_{\calI_n^{}}^{}D_n^{}\Pi_{\calI_n^{}}^\rmT)\begin{bmatrix}W_n&0\\0&\rmI\end{bmatrix}\Pi_{\calJ_n^{}}.
\end{equation}
To continue simplifying~\eqref{equation.proof of explicit frame construction 16}, note that $\Pi_{\calI_n^{}}^{}D_n^{}\Pi_{\calI_n^{}}^\rmT$ is itself a diagonal matrix: for any $m,m'=1,\dotsc,M$, the definition of the permutation matrix $\Pi_{\calI_n}$ given in Step~B.2 gives
\begin{equation*}
(\Pi_{\calI_n^{}}^{}D_n^{}\Pi_{\calI_n^{}}^\rmT)(m,m')
=\ip{D_n^{}\Pi_{\calI_n^{}}^{\rmT}\delta_{m'}}{\Pi_{\calI_n^{}}^{\rmT}\delta_m}
=\ip{D_n^{}\delta_{\pi_{\calI_{n}}^{-1}(m')}}{\delta_{\pi_{\calI_{n}}^{-1}(m)}}
=\left\{\begin{array}{ll}\lambda_{n;\pi_{\calI_n}^{-1}(m)},&m=m',\\0,&m\neq m'.\end{array}\right.
\end{equation*}
That is, $\Pi_{\calI_n^{}}^{}D_n^{}\Pi_{\calI_n^{}}^\rmT$ is the diagonal matrix whose first $R_n$ diagonal entries \smash{$\set{\beta_r}_{r=1}^{R_n}=\set{\lambda_{n;\pi_{\calI_n}^{-1}(r)}}_{r=1}^{R_n}$} match those of the aforementioned $R_n\times R_n$ diagonal matrix $D_{n;\calI_n}$ and whose remaining $M-R_n$ diagonal entries \smash{$\set{\lambda_{n;\pi_{\calI_n}^{-1}(m)}}_{m=R_n+1}^M$} form the diagonal of an $(M-R_n)\times(M-R_n)$ diagonal matrix $D_{n;\calI_n^\rmc}$:
\begin{equation}
\label{equation.proof of explicit frame construction 17}
\Pi_{\calI_n^{}}^{}D_n^{}\Pi_{\calI_n^{}}^\rmT
=\begin{bmatrix}D_{n;\calI_n}&0\\0&D_{n;\calI_n^\rmc}\end{bmatrix}.
\end{equation}
Substituting~\eqref{equation.proof of explicit frame construction 17} into~\eqref{equation.proof of explicit frame construction 16} gives
\begin{equation}
\label{equation.proof of explicit frame construction 18}
F_{n}^{}F_{n}^{*}U_n^{}V_n^{}\Pi_{\calI_n^{}}^\rmT\begin{bmatrix}W_n&0\\0&\rmI\end{bmatrix}\Pi_{\calJ_n^{}}
=U_n^{}V_n^{}\Pi_{\calI_n^{}}^\rmT\begin{bmatrix}D_{n;\calI_n}&0\\0&D_{n;\calI_n^\rmc}\end{bmatrix}\begin{bmatrix}W_n&0\\0&\rmI\end{bmatrix}\Pi_{\calJ_n^{}}
=U_n^{}V_n^{}\Pi_{\calI_n^{}}^\rmT\begin{bmatrix}D_{n;\calI_n}W_n&0\\0&D_{n;\calI_n^\rmc}\end{bmatrix}\Pi_{\calJ_n^{}}.
\end{equation}
Meanwhile, to simplify the second term in \eqref{equation.proof of explicit frame construction 15}, we recall the definition of $f_{n+1}$ from Step~B.4:
\begin{equation}
\label{equation.proof of explicit frame construction 19}
f_{n+1}^{}f_{n+1}^{*}U_n^{}V_n^{}\Pi_{\calI_n^{}}^\rmT\begin{bmatrix}W_n&0\\0&\rmI\end{bmatrix}\Pi_{\calJ_n^{}}
=U_n^{}V_n^{}\Pi_{\calI_n}^{\rmT}\begin{bmatrix}v_n\\0\end{bmatrix}\begin{bmatrix}v_n^\rmT&0\end{bmatrix}\begin{bmatrix}W_n&0\\0&\rmI\end{bmatrix}\Pi_{\calJ_n^{}}
=U_n^{}V_n^{}\Pi_{\calI_n}^{\rmT}\begin{bmatrix}v_n^{}v_n^\rmT W_n^{}&0\\0&0\end{bmatrix}\Pi_{\calJ_n^{}}.
\end{equation}
Substituting~\eqref{equation.proof of explicit frame construction 18} and~\eqref{equation.proof of explicit frame construction 19} into~\eqref{equation.proof of explicit frame construction 15}, simplifying the result, and recalling~\eqref{equation.proof of explicit frame construction 4} gives
\begin{equation*}
F_{n+1}^{}F_{n+1}^{*}U_{n+1}^{}
=U_n^{}V_n^{}\Pi_{\calI_n^{}}^\rmT\begin{bmatrix}(D_{n;\calI_n}^{}+v_n^{}v_n^\rmT)W_n^{}&0\\0&D_{n;\calI_n^\rmc}\end{bmatrix}\Pi_{\calJ_n^{}}
=U_n^{}V_n^{}\Pi_{\calI_n^{}}^\rmT\begin{bmatrix}W_nD_{n+1;\calJ_n}&0\\0&D_{n;\calI_n^\rmc}\end{bmatrix}\Pi_{\calJ_n^{}}.
\end{equation*}
By introducing an extra permutation matrix and its inverse and recalling the definition of $U_{n+1}$, this simplifies to
\begin{equation}
\label{equation.proof of explicit frame construction 20}
F_{n+1}^{}F_{n+1}^{*}U_{n+1}^{}
=U_n^{}V_n^{}\Pi_{\calI_n^{}}^\rmT\begin{bmatrix}W_n&0\\0&\rmI\end{bmatrix}\Pi_{\calJ_n^{}}^{}\Pi_{\calJ_n^{}}^\rmT\begin{bmatrix}D_{n+1;\calJ_n}&0\\0&D_{n;\calI_n^\rmc}\end{bmatrix}\Pi_{\calJ_n^{}}
=U_{n+1}\Pi_{\calJ_n^{}}^\rmT\begin{bmatrix}D_{n+1;\calJ_n}&0\\0&D_{n;\calI_n^\rmc}\end{bmatrix}\Pi_{\calJ_n^{}}.
\end{equation}
We now partition the $\set{\lambda_{n+1;m}}_{m=1}^{M}$ of $D_{n+1}$ into $\calJ_n$ and $\calJ_n^\rmc$ and mimic the derivation of~\eqref{equation.proof of explicit frame construction 17}, writing $D_{n+1}$ in terms of \smash{$D_{n+1;\calJ_n}$} and \smash{$D_{n+1;\calJ_n^\rmc}$}.  Note here that by the manner in which $\calI_n$ and $\calJ_n$ were constructed, the values of $\set{\lambda_{n;m}}_{m\in\calI_n^\rmc}$ are equal to those of $\set{\lambda_{n+1;m}}_{\calJ_n^\rmc}$, as the two sets represent exactly those values which are common to both $\set{\lambda_{n;m}}_{m=1}^{M}$ and $\set{\lambda_{n+1;m}}_{m=1}^{M}$.  As these two sequences are also both in nonincreasing order, we have $D_{n;\calI_n^\rmc}=D_{n+1;\calJ_n^\rmc}$ and so
\begin{equation}
\label{equation.proof of explicit frame construction 21}
\Pi_{\calJ_n^{}}D_{n+1}^{}\Pi_{\calJ_n^{}}^\rmT
=\begin{bmatrix}D_{n+1;\calJ_n}&0\\0&D_{n+1;\calJ_n^\rmc}\end{bmatrix}
=\begin{bmatrix}D_{n+1;\calJ_n}&0\\0&D_{n;\calI_n^\rmc}\end{bmatrix} \, .
\end{equation}
Substituting~\eqref{equation.proof of explicit frame construction 21} into~\eqref{equation.proof of explicit frame construction 20} yields \smash{$F_{n+1}^{}F_{n+1}^{*}U_{n+1}^{}=U_{n+1}^{}D_{n+1}^{}$}, completing this direction of the proof.

($\Rightarrow$) Let $\set{\lambda_m}_{m=1}^{M}$ and $\set{\mu_n}_{n=1}^{N}$ be any nonnegative nonincreasing sequences, and let $F=\set{f_n}_{n=1}^N$ be any sequence of vectors whose frame operator $FF^*$ has $\set{\lambda_m}_{m=1}^{M}$ as its spectrum and has $\norm{f_n}^2=\mu_n$ for all $n=1,\dotsc,N$.  We will show that this $F$ can be constructed by following Step~A and Step~B of this result.  To see this, for any $n=1,\dotsc,N$, let $F_n=\set{f_{n'}}_{n'=1}^{n}$ and let $\set{\lambda_{n;m}}_{m=1}^{M}$ be the spectrum of the corresponding frame operator $F_n^{}F_n^{*}$.  Letting $\lambda_{0;m}:=0$ for all $m$, the proof of Theorem~\ref{theorem.necessity and sufficiency of eigensteps} demonstrated that the sequence of spectra $\set{\set{\lambda_{n;m}}_{m=1}^{M}}_{n=0}^{N}$ necessarily forms a sequence of eigensteps as specified by Definition~\ref{definition.eigensteps}.  This particular set of eigensteps is the one we choose in Step~A.

All that remains to be shown is that we can produce our specific $F$ by using Step~B.  Here, we must carefully exploit our freedom to pick $U_1$ and the $V_n$'s; the proper choice of these unitary matrices will result in $F$, while other choices will produce other sequences of vectors that are only related to $F$ through a potentially complicated series of rotations.  Indeed, note that since $\set{\set{\lambda_{n;m}}_{m=1}^{M}}_{n=0}^{N}$ is a valid sequence of eigensteps, then the other direction of this proof, as given earlier, implies that any choice of $U_1$ and $V_n$'s will result in a sequence of vectors whose eigensteps match those of $F$.  Moreover, quantities that we considered in the other direction of the proof that only depended on the choice of eigensteps, such as $\calI_n$, $\calJ_n$, $\set{\beta_r}_{r=1}^{R_n}$, $\set{\gamma_r}_{r=1}^{R_n}$, etc., are thus also well-defined in this direction; in the following arguments, we recall several such quantities and make further use of their previously-derived properties.

To be precise, let $U_1$ be any one of the infinite number of unitary matrices whose first column $u_{1;1}$ satisfies $f_1=\sqrt{\mu_1}u_{1;1}$.  We now proceed by induction, assuming that for any given $n=1,\dotsc,N-1$, we have followed Step~B and have made appropriate choices for $\set{V_{n'}}_{n'=1}^{n-1}$ so as to correctly produce $F_n=\set{f_{n'}}_{n'=1}^{n}$; we show how the appropriate choice of $V_n$ will correctly produce $f_{n+1}$.  To do so, we again write the $n$th spectrum $\set{\lambda_{n;m}}_{m=1}^{M}$ in terms of its multiplicities as \smash{$\{\lambda_{n;m(k,l)}\}_{k=1}^{K_n}\,_{l=1}^{L_{\smash{n;k}}}$}.  For any $k=1,\dotsc,K_n$, Step~B of Theorem~\ref{theorem.necessity and sufficiency of eigensteps} gives that the norm of the projection of $f_{n+1}$ onto the $k$th eigenspace of $F_n^{}F_n^*$ is necessarily given by
\begin{equation}
\label{equation.proof of explicit frame construction 22}
\norm{P_{n;\lambda_{n;m(k,1)}}f_{n+1}}^2=-\lim_{x\rightarrow\lambda_{n;m(k,1)}}(x-\lambda_{n;m(k,1)})\frac{p_{n+1}(x)}{p_n(x)} \, ,
\end{equation} 
where $p_{n}(x)$ and $p_{n+1}(x)$ are defined by~\eqref{equation.proof of explicit frame construction 1}.  Note that by picking $l=1$, $\lambda_{n;m(k,1)}$ represents the first appearance of that particular value in $\set{\lambda_{n;m}}_{m=1}^{M}$.  As such, these indices are the only ones that are eligible to be members of the set $\calI_n$ found in Step~B.2.  That is, $\calI_n\subseteq\set{m(k,1): k=1,\dotsc,K_n}$.  However, these two sets of indices are not necessarily equal, since $\calI_m$ only contains $m$'s of the form $m(k,1)$ that satisfy the additional property that the multiplicity of $\lambda_{n;m}$ as a value in $\set{\lambda_{n;m'}}_{m'=1}^{M}$ exceeds its multiplicity as a value in $\set{\lambda_{n+1;m}}_{m=1}^{M}$.  To be precise, for any given $k=1,\dotsc,K_n$, if $m(k,1)\in\calI_n^\rmc$ then $\lambda_{n;m(k,1)}$ appears as a root of $p_{n+1}(x)$ at least as many times as it appears as a root of $p_n(x)$, meaning in this case that the limit in \eqref{equation.proof of explicit frame construction 22} is necessarily zero.  If, on the other hand, $m(k,1)\in\calI_n$, then writing $\pi_{\calI_n}(m(k,1))$ as some $r\in\set{1,\dotsc,R_n}$ and recalling the definitions of $b(x)$ and $c(x)$ in~\eqref{equation.proof of explicit frame construction 1} and $v(r)$ in~\eqref{equation.proof of explicit frame construction 3}, we can rewrite \eqref{equation.proof of explicit frame construction 22} as
\begin{equation}
\label{equation.proof of explicit frame construction 23}
\norm{P_{n;\beta_r}f_{n+1}}^2
=-\lim_{x\rightarrow\beta_r}(x-\beta_r)\frac{p_{n+1}(x)}{p_n(x)}
=-\lim_{x\rightarrow\beta_r}(x-\beta_r)\frac{c(x)}{b(x)}
=-\,\frac{\displaystyle\prod_{r''=1}^{R_n} (\beta_r-\gamma_{r''})}{\displaystyle\prod_{\substack{r''=1\\r''\neq r}}^{R}(\beta_r-\beta_{r''})}
=[v_n(r)]^2.
\end{equation}
As such, we can write $f_{n+1}$ as
\begin{equation}
\label{equation.proof of explicit frame construction 24}
f_{n+1}
=\sum_{k=1}^{K_n}P_{n;\lambda_{n;m(k,1)}}f_{n+1}
=\sum_{r=1}^{R_n}P_{n;\beta_r}f_{n+1}
=\sum_{r=1}^{R_n}v_n(r)\frac1{v_n(r)}P_{n;\beta_r}f_{n+1}
=\sum_{m\in\calI_n}v_n(\pi_{\calI_n}(m))\frac1{v_n(\pi_{\calI_n}(m))}P_{n;\beta_{\pi_{\calI_n}(m)}}f_{n+1}
\end{equation}
where each $\frac1{v_n(\pi_{\calI_n}(m))}P_{n;\beta_{\pi_{\calI_n}(m)}}f_{n+1}$ has unit norm by~\eqref{equation.proof of explicit frame construction 23}.  We now pick a new orthonormal eigenbasis $\hat{U}_n:=\set{\hat{u}_{n;m}}_{m=1}^{M}$ for $F_n^{}F_n^*$ that has the property that for any $k=1,\dotsc,K_n$, both \smash{$\set{u_{n;m(k,l)}}_{l=1}^{L_{\smash{n;k}}}$} and \smash{$\set{\hat{u}_{n;m(k,l)}}_{l=1}^{L_{\smash{n;k}}}$} span the same eigenspace and, for every $m(k,1)\in\calI_n$, has the additional property that \smash{$\hat{u}_{n;m(k,1)}=\frac1{v_n(\pi_{\calI_n}(m(k,1)))}P_{n;\beta_{\pi_{\calI_n}(m(k,1))}}f_{n+1}$}.  As such, \eqref{equation.proof of explicit frame construction 24} becomes
\begin{equation}
\label{equation.proof of explicit frame construction 25}
f_{n+1}
=\sum_{m\in\calI_n}v_n(\pi_{\calI_n}(m))\hat{u}_{n;m}
=\hat{U}_n^{}\sum_{m\in\calI_n}v_n(\pi_{\calI_n}(m))\delta_m
=\hat{U}_n^{}\sum_{r=1}^{R_n}v_n(r)\delta_{\pi_{\calI_n}^{-1}(r)}
=\hat{U}_n^{}\Pi_{\calI_n}^{\rmT}\sum_{r=1}^{R_n}v_n(r)\delta_r
=\hat{U}_n^{}\Pi_{\calI_n}^{\rmT}\begin{bmatrix}v_n\\0\end{bmatrix}.
\end{equation}
Letting $V_n$ be the unitary matrix \smash{$V_n=U_n^*\hat{U}_n^{}$}, the eigenspace spanning condition gives that $V_n$ is block-diagonal whose $k$th diagonal block is of size $L_{n;k}\times L_{n;k}$.  Moreover, with this choice of $V_n$, \eqref{equation.proof of explicit frame construction 25} becomes
\begin{equation*}
f_{n+1}
=U_n^{}U_n^*\hat{U}_n^{}\Pi_{\calI_n}^{\rmT}\begin{bmatrix}v_n\\0\end{bmatrix}
=U_n^{}V_n^{}\Pi_{\calI_n}^{\rmT}\begin{bmatrix}v_n\\0\end{bmatrix}
\end{equation*}
meaning that $f_{n+1}$ can indeed be constructed by following Step~B.
\end{proof}


\section*{Acknowledgments}
The authors thank Prof.~Peter G.~Casazza for insightful discussions.  This work was supported by NSF DMS 1042701, NSF DMS 1008183, NSF CCF 1017278, AFOSR F1ATA01103J001, AFOSR F1ATA00183G003 and the A.~B.~Krongard Fellowship.  The views expressed in this article are those of the authors and do not reflect the official policy or position of the United States Air Force, Department of Defense, or the U.S.~Government.


\appendix

\section{Proof of Lemma~\ref{lemma.interlacing and nonpositive limits}}

\noindent
($\Rightarrow$)
Let $\set{\gamma_m}_{m=1}^{M}$ interlace on $\set{\beta_m}_{m=1}^{M}$, and let $\lambda=\beta_m$ for some $m=1,\dotsc,M$.  Letting $L_p$ denote the multiplicity of $\lambda$ as a root of $p(x)$, the fact that $\set{\beta_m}_{m=1}^{M}\sqsubseteq\set{\gamma_m}_{m=1}^{M}$ implies that the multiplicity $L_q$ of $\lambda$ as a root of $q(x)$ is at least $L_p-1$.  Moreover, if $L_q>L_p-1$ then our claim holds at $\lambda$ since
\begin{equation*} 
\lim_{x\rightarrow\lambda}(x-\lambda)\frac{q(x)}{p(x)}=0\leq 0.
\end{equation*}
Meanwhile, if $L_q=L_p-1$, then choosing $m_p=\min\set{m: \beta_m=\lambda}$ gives
\begin{equation}
\label{equation.proof of interlacing and nonpositive limits 1}
\begin{array}{llrcccl}
\beta_m>\lambda,&&1			&\leq&m	&\leq&m_p-1,\\
\beta_m=\lambda,&&m_p		&\leq&m	&\leq&m_p+L_p-1,\\
\beta_m<\lambda,&&m_p+L_p	&\leq&m	&\leq&M.
\end{array}
\end{equation}
We now determine a similar set of relations between $\lambda$ and all choices of $\gamma_m$.  For $1\leq m\leq m_p-1$, interlacing and~\eqref{equation.proof of interlacing and nonpositive limits 1} imply $\gamma_m\geq\beta_m>\lambda$.  If instead $m_p+1\leq m\leq m_p+L_p-1$, then interlacing and~\eqref{equation.proof of interlacing and nonpositive limits 1} imply $\lambda=\beta_m\leq\gamma_m\leq\beta_{m-1}=\lambda$ and so $\gamma_m=\lambda$.  Another possibility is to have $m_p+L_p+1\leq m\leq M$, in which case interlacing and~\eqref{equation.proof of interlacing and nonpositive limits 1} imply $\gamma_m\leq\beta_{m-1}<\lambda$.  Taken together, we have
\begin{equation}
\label{equation.proof of interlacing and nonpositive limits 2}
\begin{array}{llrcccl}
\gamma_m>\lambda,&&1			&\leq&m	&\leq&m_p-1,\\
\gamma_m=\lambda,&&m_p+1		&\leq&m	&\leq&m_p+L_p-1,\\
\gamma_m<\lambda,&&m_p+L_p+1	&\leq&m	&\leq&M.
\end{array}
\end{equation}
Note that the table~\eqref{equation.proof of interlacing and nonpositive limits 2} is unlike~\eqref{equation.proof of interlacing and nonpositive limits 1} in that in~\eqref{equation.proof of interlacing and nonpositive limits 2}, the relationship between $\gamma_m$ and $\lambda$ is still undecided for $m=m_p$ and $m=m_p+L_p$.  Indeed, in general we only know $\gamma_{m_p}\geq\beta_{m_p}=\lambda$, and so either $\gamma_{m_p}=\lambda$ or $\gamma_{m_p}>\lambda$.  Similarly, we only know $\gamma_{m_p+L_p}\leq\beta_{m_p+L_p-1}=\lambda$, so either $\gamma_{m_p+L_p}=\lambda$ or $\gamma_{m_p+L_p}<\lambda$.  Of these four possibilities, three lead to either having $L_q=L_p+1$ or $L_q=L_p$; only the case where $\gamma_{m_p}>\lambda$ and $\gamma_{m_p+L_p}<\lambda$ leads to our current assumption that $L_q=L_p-1$.  As such, under this assumption \eqref{equation.proof of interlacing and nonpositive limits 2} becomes
\begin{equation}
\label{equation.proof of interlacing and nonpositive limits 3}
\begin{array}{llrcccl}
\gamma_m>\lambda,&&1		&\leq&m	&\leq&m_p,\\
\gamma_m=\lambda,&&m_p+1	&\leq&m	&\leq&m_p+L_p-1,\\
\gamma_m<\lambda,&&m_p+L_p	&\leq&m	&\leq&M.
\end{array}
\end{equation}
We now prove our claim using~\eqref{equation.proof of interlacing and nonpositive limits 1} and~\eqref{equation.proof of interlacing and nonpositive limits 3}:
\begin{equation*} 
\lim_{x\rightarrow\lambda}(x-\lambda)\frac{q(x)}{p(x)}
=\lim_{x\rightarrow\lambda}\frac{\displaystyle(x-\lambda)^{L_p}\prod_{m=1}^{m_p}(x-\gamma_m)\prod_{m=m_p+L_p}^{M} (x-\gamma_m)}{\displaystyle(x-\lambda)^{L_p}\prod_{m=1}^{m_p-1} (x-\beta_m)\prod_{m=m_p+L_p}^{M} (x-\beta_m)}
=\frac{\displaystyle\prod_{m=1}^{m_p}(\lambda-\gamma_m)\prod_{m=m_p+L_p}^{M} (\lambda-\gamma_m)}{\displaystyle\prod_{m=1}^{m_p-1} (\lambda-\beta_m)\prod_{m=m_p+L_p}^{M} (\lambda-\beta_m)}<0.
\end{equation*}
($\Leftarrow$)  We prove by induction on $M$.  For $M=1$, we have $p(x)=x-\beta_1$ and $q(x)=x-\gamma_1$, and so if
\begin{equation*}
0
\geq\lim_{x\rightarrow\beta_1}(x-\beta_1)\frac{q(x)}{p(x)}
=\lim_{x\rightarrow\beta_1}(x-\beta_1)\frac{(x-\gamma_1)}{(x-\beta_1)}
=\lim_{x\rightarrow\beta_1}(x-\gamma_1)
=\beta_1-\gamma_1 \, ,
\end{equation*}
then $\beta_1\leq\gamma_1$, and so $\set{\gamma_1}$ interlaces on $\set{\beta_1}$, as claimed.  Now assume this direction of the proof holds for $M'=1,\dotsc,M-1$, and let $\set{\beta_m}_{m=1}^{M}$ and $\set{\gamma_m}_{m=1}^{M}$ be real, nonincreasing and have the property that
\begin{equation}
\label{equation.proof of interlacing and nonpositive limits 4}
\lim_{x\rightarrow\beta_m}(x-\beta_m)\frac{q(x)}{p(x)}\leq0 \qquad \forall m=1,\dotsc,M,
\end{equation}
where $p(x)$ and $q(x)$ are defined as in the statement of the result.  We will show that $\set{\gamma_m}_{m=1}^{M}$ interlaces on $\set{\beta_m}_{m=1}^{M}$.

To do this, we consider two cases.  The first case is when $\set{\beta_m}_{m=1}^{M}$ and $\set{\gamma_m}_{m=1}^{M}$ have no common members, that is, $\beta_m\neq\gamma_{m'}$ for all $m,m'=1,\dotsc,M$.  In this case, note that if $\beta_m=\beta_{m'}$ for some $m\neq m'$ then the corresponding limit in~\eqref{equation.proof of interlacing and nonpositive limits 4} would diverge, contradicting our implicit assumption that these limits exist and are nonpositive.  As such, in this case the values of $\set{\beta_m}_{m=1}^{M}$ are necessarily distinct, at which point \eqref{equation.proof of interlacing and nonpositive limits 4} for a given $m$ becomes:
\begin{equation}
\label{equation.proof of interlacing and nonpositive limits 5}
0
\geq\lim_{x\rightarrow\beta_m}(x-\beta_m)\prod_{m'=1}^{M}\frac{(x-\gamma_{m'})}{(x-\beta_{m'})}
=\frac{\displaystyle\prod_{m'=1}^{M}(\beta_m-\gamma_{m'})}{\displaystyle\prod_{\substack{m'=1\\m'\neq m}}^{M}(\beta_m-\beta_{m'})}
=\frac{\displaystyle\prod_{m'=1}^{M}(\beta_m-\gamma_{m'})}{\displaystyle\prod_{m'=1}^{m-1}(\beta_m-\beta_{m'})\prod_{m'=m+1}^{M}(\beta_m-\beta_{m'})}.
\end{equation}
Moreover, since $\beta_m\neq\gamma_{m'}$ for all $m,m'$, then the limit in \eqref{equation.proof of interlacing and nonpositive limits 5} is nonzero.  As the sign of the denominator on the right-hand side of \eqref{equation.proof of interlacing and nonpositive limits 5} is $(-1)^{m-1}$, the sign of the corresponding numerator is
\begin{equation*}
(-1)^m
=\sgn\biggparen{\,\prod_{m'=1}^{M}(\beta_m-\gamma_{m'})}
=\sgn(q(\beta_m))
\qquad \forall m=1,\dotsc,M.
\end{equation*}
Thus, for any $m=2,\dotsc,M$, $q(x)$ changes sign over $[\beta_m,\beta_{m-1}]$, implying by the Intermediate Value Theorem that at least one of the roots $\set{\gamma_m}_{m=1}^{M}$ of $q(x)$ lies in $(\beta_m,\beta_{m-1})$.  Moreover, since $q(x)$ is monic, we have $\lim_{x\rightarrow\infty}q(x)=\infty$; coupled with the fact that $q(\beta_1)<0$, this implies that at least one root of $q(x)$ lies in $(\beta_1,\infty)$.  Thus, each of the $M$ disjoint subintervals of \smash{$(\beta_1,\infty)\cup\bigbracket{\cup_{m=2}^{M}(\beta_m,\beta_{m-1})}$} contains at least one of the $M$ roots of $q(x)$.  This is only possible if each of these subintervals contains exactly one of these roots.  Moreover, since $\set{\gamma_m}_{m=1}^{M}$ is nonincreasing, this implies $\beta_1<\gamma_1$ and $\beta_m<\gamma_m<\beta_{m-1}$ for all $m=2,\dotsc,M$, meaning that $\set{\gamma_m}_{m=1}^{M}$ indeed interlaces on $\set{\beta_m}_{m=1}^{M}$.

We are thus left to consider the remaining case where $\set{\beta_m}_{m=1}^{M}$ and $\set{\gamma_m}_{m=1}^{M}$ share at least one common member.  Fix $\lambda$ such that $\beta_m=\lambda=\gamma_{m'}$ for at least one pair $m,m'=1,\dotsc,M$.  Let $m_p=\min\set{m: \beta_m=\lambda}$ and $m_q=\min\set{m: \gamma_m=\lambda}$.  Let $P(x)$ and $Q(x)$ be $(M-1)$-degree polynomials such that $p(x)=(x-\lambda)P(x)$ and $q(x)=(x-\lambda)Q(x)$.  Here, our assumption~\eqref{equation.proof of interlacing and nonpositive limits 4} implies
\begin{equation}
\label{equation.proof of interlacing and nonpositive limits 7}
0
\geq\lim_{x\rightarrow\beta_m}(x-\beta_m)\frac{q(x)}{p(x)}
=\lim_{x\rightarrow\beta_m}(x-\beta_m)\frac{(x-\lambda)Q(x)}{(x-\lambda)P(x)}
=\lim_{x\rightarrow\beta_m}(x-\beta_m)\frac{Q(x)}{P(x)} \qquad \forall m=1,\dotsc,M.
\end{equation}
Since $P(x)$ and $Q(x)$ satisfy~\eqref{equation.proof of interlacing and nonpositive limits 7} and have degree $M-1$, our inductive hypothesis gives that the roots $\set{\gamma_m}_{m\neq m_q}$of $Q(x)$ interlace on the roots $\set{\beta_m}_{m\neq m_p}$ of $P(x)$.

We claim that $m_q$ is necessarily either $m_p$ or $m_p+1$, that is, $m_p\leq m_q\leq m_p+1$.  We first show that $m_p\leq m_q$, a fact which trivially holds for $m_p=1$.  For $m_p>1$, the fact that $\set{\beta_m}_{m\neq m_p}\sqsubseteq\set{\gamma_m}_{m\neq m_q}$ implies that the value of the $(m_p-1)$th member of $\set{\gamma_m}_{m\neq m_q}$ is at least that of the $(m_p-1)$th member of $\set{\beta_m}_{m\neq m_p}$.  That is,  the $(m_p-1)$th member of $\set{\gamma_m}_{m\neq m_q}$ is at least $\beta_{m_p-1}>\lambda$, meaning $m_p-1\leq m_q-1$ and so $m_p\leq m_q$, as claimed.  We similarly prove that $m_q\leq m_p+1$, a fact which trivially holds for $m_p=M$.  For $m_p<M$, interlacing implies that the $m_p$th member of $\set{\beta_m}_{m\neq m_p}$ is at least the $(m_p+1)$th member of $\set{\gamma_m}_{m\neq m_q}$.  That is, the $(m_p+1)$th member of $\set{\gamma_m}_{m\neq m_q}$ is at most $\beta_{m_p+1}\leq\lambda$ and so $m_p+1\geq m_q$, as claimed.

Now, in the case that $m_q=m_p$, the fact that $\set{\beta_m}_{m\neq m_p}\sqsubseteq\set{\gamma_m}_{m\neq m_q}$ implies that
\begin{equation}
\label{equation.proof of interlacing and nonpositive limits 8}
\beta_{M}\leq\gamma_{M}\leq\dots\leq\beta_{m_p+1}\leq\gamma_{m_p+1}\leq\beta_{m_p-1}\leq\gamma_{m_p-1}\leq\dots\leq\beta_1\leq\gamma_1.
\end{equation}
Since in this case $\gamma_{m_p+1}=\gamma_{m_q+1}\leq\lambda<\beta_{m_p-1}$, the terms $\beta_{m_p}=\lambda$ and $\gamma_{m_p}=\gamma_{m_q}=\lambda$ can be inserted into~\eqref{equation.proof of interlacing and nonpositive limits 8}:
\begin{equation*}
\beta_{M}\leq\gamma_{M}\leq\dots\leq\beta_{m_p+1}\leq\gamma_{m_p+1}\leq\lambda\leq\lambda\leq\beta_{m_p-1}\leq\gamma_{m_p-1}\leq\dots\leq\beta_1\leq\gamma_1,
\end{equation*}
and so $\set{\beta_m}_{m=1}^{M}\sqsubseteq\set{\gamma_m}_{m=1}$.  In the remaining case where $m_q=m_p+1$, having $\set{\beta_m}_{m\neq m_p}\sqsubseteq\set{\gamma_m}_{m\neq m_q}$ means that
\begin{equation}
\label{equation.proof of interlacing and nonpositive limits 9}
\beta_{M}\leq\gamma_{M}\leq\dots\leq\beta_{m_p+2}\leq\gamma_{m_p+2}\leq\beta_{m_p+1}\leq\gamma_{m_p}\leq\beta_{m_p-1}\leq\gamma_{m_p-1}\leq\dots\leq\beta_1\leq\gamma_1.
\end{equation}
Since in this case $\beta_{m_p+1}\leq\lambda<\gamma_{m_q-1}=\gamma_{m_p}$, the terms $\gamma_{m_p+1}=\gamma_{m_q}=\lambda$ and $\beta_{m_p}=\lambda$ can be inserted into~\eqref{equation.proof of interlacing and nonpositive limits 9}:
\begin{equation*}
\beta_{M}\leq\gamma_{M}\leq\dots\leq\beta_{m_p+2}\leq\gamma_{m_p+2}\leq\beta_{m_p+1}\leq\lambda\leq\lambda\leq\gamma_{m_p}\leq\beta_{m_p-1}\leq\gamma_{m_p-1}\leq\dots\leq\beta_1\leq\gamma_1
\end{equation*}
and so $\set{\beta_m}_{m=1}^{M}\sqsubseteq\set{\gamma_m}_{m=1}$ in this case as well.

\section{Proof of Lemma~\ref{lemma.equal limits}}

\noindent
Fix any $m=1,\dotsc,M$, and let $L$ be the multiplicity of $\beta_m$ as a root of $p(x)$.  Since
\begin{equation}
\label{equation.proof of equal limits 1}
\lim_{x\rightarrow\beta_m}(x-\beta_m)\frac{q(x)}{p(x)}
=\lim_{x\rightarrow\beta_m}(x-\beta_m)\frac{r(x)}{p(x)} \, ,
\end{equation}
where each of these two limits is assumed to exist, then the multiplicities of $\beta_m$ as a roots of $q(x)$ and $r(x)$ are both at least $L-1$.  As such, evaluating $l$th derivatives at $\beta_m$ gives $q^{(l)}(\beta_m)=0=r^{(l)}(\beta_m)$ for all $l=0,\dots,L-2$.  Meanwhile, for $l=L-1$, l'H\^{o}pital's Rule gives
\begin{equation}
\label{equation.proof of equal limits 2}
\lim_{x\rightarrow\beta_m}(x-\beta_m)\frac{q(x)}{p(x)}
=\lim_{x\rightarrow\beta_m}\frac{q(x)}{(x-\beta_m)^{L-1}}\frac{(x-\beta_m)^L}{p(x)}
=\frac{q^{(L-1)}(\beta_m)}{(L-1)!}\frac{L!}{p^{(L)}(\beta_m)}
=\frac{Lq^{(L-1)}(\beta_m)}{p^{(L)}(\beta_m)}.
\end{equation}
Deriving a similar expression for $r(x)$ and substituting both it and~\eqref{equation.proof of equal limits 2} into~\eqref{equation.proof of equal limits 1} yields $q^{(L-1)}(\beta_m)=r^{(L-1)}(\beta_m)$.   As such, $q^{(l)}(\beta_m)=r^{(l)}(\beta_m)$ for all $l=0,\dots,L-1$.  As this argument holds at every distinct $\beta_m$, we see that $q(x)-r(x)$ has $M$ roots, counting multiplicity.  But since $q(x)$ and $r(x)$ are both monic, $q(x)-r(x)$ has degree at most $M-1$ and so $q(x)-r(x)\equiv 0$, as claimed.

\end{document}